\documentclass[11pt,fleqn]{article}
\usepackage{amsfonts}
\usepackage{amssymb}
\usepackage{amsthm}

\usepackage[leqno]{amsmath}
\usepackage{mathtools}

\makeatletter
\newcommand{\leqnomode}{\tagsleft@true}
\newcommand{\reqnomode}{\tagsleft@false}
\makeatother

\usepackage[UKenglish]{babel}
\usepackage{enumerate}
\usepackage{hyperref}
\usepackage{latexsym}
\usepackage{lmodern}
\theoremstyle{plain}
\newtheorem{theorem}{Theorem}[section]

\newtheorem{lemma}[theorem]{Lemma}
\newtheorem{observation}[theorem]{Observation}

\theoremstyle{remark}

\usepackage{tikz}
\usetikzlibrary{arrows,snakes,backgrounds}
\usetikzlibrary{decorations.pathmorphing}
\usepackage{graphicx}
\usetikzlibrary{positioning}
\usepackage{enumitem}

\theoremstyle{plain}
\newtheorem{innercustomthm}{Theorem}
\newenvironment{customtheorem}[1]
  {\renewcommand\theinnercustomthm{#1}\innercustomthm}
  {\endinnercustomthm}

\def\d{\hbox{-}}

\newcommand{\sset}[1]{\left\{#1\right\}}

\usepackage[margin=2.25cm]{geometry}
\usepackage{marvosym}
\renewcommand{\tilde}[1]{\widetilde{#1}}

\def\longbox#1{\parbox{0.85\textwidth}{#1}}

\newcommand*\samethanks[1][\value{footnote}]{\footnotemark[#1]}

\title{Strongly Perfect Claw-free Graphs - A Short Proof}

\author{Maria Chudnovsky \thanks{Supported by NSF Grant DMS-1763817. This material is based upon work supported by, or in part by, the U.S. Army Research Laboratory and the U. S. Army Research Office under grant number W911NF-16-1-0404.}
\\
Princeton University, Princeton, NJ 08544
\medskip\\
Cemil Dibek \samethanks
\\
Princeton University, Princeton, NJ 08544}

\date{\today}
\begin{document}\maketitle

\vspace{-0.3cm}

\begin{abstract}
A graph is \emph{strongly perfect} if every induced subgraph $H$ of it has a stable set that meets every maximal clique of $H$. A graph is \emph{claw-free} if no vertex has three pairwise non-adjacent neighbors. The characterization of claw-free graphs that are strongly perfect by a set of forbidden induced subgraphs was conjectured by Ravindra \cite{Ravindra2} in 1990 and was proved by Wang \cite{Wang} in 2006. Here we give a shorter proof of this characterization.
\end{abstract}

\section{Introduction} \label{sec:intro}

All graphs in this paper are finite and simple.  Let $G = (V, E)$ be a graph. For $X \subseteq V$, $G|X$ denotes the subgraph of $G$ induced by the vertex set $X$. A set $X \subseteq V$ is \emph{anticonnected} if $X \neq \emptyset$ and the graph $G^C | X$ is connected (here $G^C$ denotes the complement of $G$). A graph $G$ is anticonnected if $G^C$ is connected. An \emph{anticomponent} of a set $S \subseteq V$ is a maximal anticonnected subset of $S$. An anticomponent $D$ is \emph{non-trivial} if $|D| \geq 2$. We say that $G$ \emph{contains} a graph $H$ if $G$ has an induced subgraph isomorphic to $H$. For a graph $H$, $X \subseteq V$ is an $H$ in $G$ if $G|X$ is isomorphic to $H$.

For a vertex $v \in V$, we let $N_G(v) = N(v)$ denote the set of neighbors of $v$ in $G$, and we write $N[v] = N(v) \cup \{v\}$. We define the neighborhood of a set $U \subseteq V$ as $N(U) = \{v \in V \setminus U \: | \: uv \in E \text{ for some } u \in U\}$ and we write $N[U] = N(U) \cup U$. 

Two sets $X, Y \subseteq V$ are \emph{complete} to each other if they are disjoint and every vertex in $X$ is adjacent to every vertex in $Y$, and \emph{anticomplete} to each other if they are disjoint and no vertex in $X$ is adjacent to a vertex in $Y$. We say that $v$ is \emph{complete} (\emph{anticomplete}) to $X \subseteq V$ if $\sset{v}$ is complete (anticomplete) to $X$. A vertex $v \in V \setminus X$ that is neither complete nor anticomplete to $X$ is \emph{mixed} on X.

A \emph{clique} in $G$ is a set of pairwise adjacent vertices, and a \emph{stable set} is a set of pairwise non-adjacent vertices. A \emph{maximal clique} is a clique that is not a subset of a larger clique. A stable set in $G$ is \emph{strong} if it meets every maximal clique of $G$. A vertex $v \in V$ is \emph{simplicial} if $N(v)$ is a clique. An edge $uv \in E$ is \emph{simplicial} if $N(u)$ is complete to $N(v)$. We say that $uv$ is a {\em cosimplicial non-edge} in $G$ if $uv$ is a simplicial edge in $G^C$. A \emph{simplicial clique} in $G$ is a non-empty clique $K$ such that for every $k \in K$ the set of neighbors of $k$ in $V \setminus K$ is a clique. A \emph{matching} is a set of edges no two of which share a common vertex.

A {\em path} in $G$ is a sequence of distinct vertices $p_0 \d p_1 \d \dots \d p_k$ where $p_i$ is adjacent to $p_j$ if and only if $|i-j|=1$, and the \emph{length} of a path is the number of edges in it. A path is {\em odd} if its length is odd, and {\em even} otherwise. For a path $P$ with endpoints $a, b$, the \emph{interior} of $P$, denoted by $P^*$, is the set $V(P) \setminus \{a, b\}$. An {\em antipath} is an induced subgraph whose complement is a path. The length of an antipath is the number of edges in its complement. When we say $P = p_1 \d p_2 \d \dots \d p_\ell$ is a path from a vertex $p_1$ to a set $X$, we mean that $V(P) \cap X = \emptyset$, and $V(P) \setminus \{p_\ell\}$ is anticomplete to $X$, and $p_\ell$ has neighbors in $X$. An \emph{even pair} is a pair of vertices $\{u, v\}$ such that every path from $u$ to $v$ is even, and in particular, $u$ and $v$ are non-adjacent. We call a set of vertices \emph{consistent} if every pair of its vertices is even.  Thus, a consistent set is stable.

Let $k \geq 4$ be an integer.  A {\em hole of length $k$} in $G$ is an induced subgraph isomorphic to the $k$-vertex cycle $C_k$, and an {\em antihole of length $k$} is an induced subgraph isomorphic to $C_k^C$. A hole (or antihole) is {\em odd} if its length is odd, and {\em even} if its length is even. A \emph{square} is a hole of length four.

A \emph{star} $S_k$ is the complete bipartite graph $K_{1,k}$. In a star $S_k$ with $k \geq 2$, the unique vertex of degree $k$ is the \emph{center} of the star and the other vertices are its \emph{leaves}. The star $S_3$ is called a \emph{claw}. A graph is \emph{claw-free} if it contains no induced claw.

A graph $G$ is \emph{perfect} if every induced subgraph $H$ of $G$ satisfies $\chi(H) = \omega(H)$, where $\chi(H)$ is the chromatic number of $H$ and $\omega(H)$ is the maximum clique size in $H$. Claude Berge introduced perfect graphs and his conjecture (now the Strong Perfect Graph Theorem) was solved by Chudnovsky et al.:

\begin{theorem}\cite{SPGT}
\label{thm:SPGT}
A graph is perfect if and only if it does not contain an odd hole or an odd antihole.
\end{theorem}

A graph is \emph{strongly perfect} if every induced subgraph of it has a strong stable set. Strongly perfect graphs have been studied by several authors (\cite{BD}, \cite{Ravindra}, \cite{Wang}) and they form an interesting class of perfect graphs, as the complement of a strongly perfect graph is not necessarily strongly perfect, unlike a perfect graph. Although there are many results concerning strongly perfect graphs (see \cite{Ravindra3} for a summary), there is no elegant characterization of strongly perfect graphs in terms of forbidden subgraphs. In \cite{Ravindra2}, the author presents a conjecture to characterize strongly perfect graphs in terms of forbidden induced subgraphs (originally posed by Berge). In the same paper, a characterization of claw-free strongly perfect graphs by five infinite families of forbidden induced subgraphs was also conjectured, and this was proved in \cite{Wang}. The main result of the current paper is a new shorter and quite different proof of this characterization.

This paper is organized as follows. In Section \ref{sec:preliminaries}, we give further definitions and some known results about claw-free graphs. We also state the main theorem, Theorem \ref{thm:stronger}, and give a brief overview of the proof. In Section \ref{sec:lemmas}, we state some observations and prove a few lemmas that will be used later. Next, in Section \ref{sec:minimal_counter}, we prove several statements about the structure of a minimal counterexample to Theorem \ref{thm:stronger}. The proof of the main theorem occupies Section \ref{sec:proof}.

\section{Preliminaries} \label{sec:preliminaries}

We start with definitions of some graphs needed to state the main theorem. A \emph{prism} is a graph consisting of two vertex-disjoint triangles $\{a_1, a_2, a_3\}$, $\{b_1, b_2, b_3\}$, and three paths $P_1, P_2, P_3$, where each $P_i$ has endpoints $a_i, b_i$, and for $1 \leq i < j \leq 3$ the only edges between $V(P_i)$ and $V(P_j)$ are $a_ia_j$ and $b_ib_j$. Thus, the interiors of $P_1, P_2, P_3$ are pairwise anticomplete. A prism is \emph{odd} if the three paths $P_1, P_2, P_3$ are odd. A \emph{handcuff} is a graph formed by connecting two vertex-disjoint even cycles via an odd path whose endpoints form a triangle with the cycles. An \emph{eye mask} is a graph consisting of two vertex-disjoint even cycles and a clique of size four formed by making an edge of each cycle complete to each other (see Figure \ref{fig:handcuff_eyemask}, where the dotted lines represent odd paths).

\begin{figure}[h]
\begin{center}
\begin{tikzpicture}[scale=0.24]

\node[inner sep=2.5pt, fill=black, circle] at (-4, 4)(v1){}; 
\node[inner sep=2.5pt, fill=black, circle] at (-4, -4)(v2){}; 
\node[inner sep=2.5pt, fill=black, circle] at (-2, 2.6)(v3){}; 
\node[inner sep=2.5pt, fill=black, circle] at (-2, -2.6)(v4){};

\node[inner sep=2.5pt, fill=black, circle] at (0, 0)(v5){}; 
\node[inner sep=2.5pt, fill=black, circle] at (3, 0)(v6){}; 
\node[inner sep=2.5pt, fill=black, circle] at (6, 0)(v7){}; 
\node[inner sep=2.5pt, fill=black, circle] at (9, 0)(v8){};

\node[inner sep=2.5pt, fill=black, circle] at (11, 2.6)(v9){}; 
\node[inner sep=2.5pt, fill=black, circle] at (11, -2.6)(v10){}; 
\node[inner sep=2.5pt, fill=black, circle] at (13, 4)(v11){}; 
\node[inner sep=2.5pt, fill=black, circle] at (13, -4)(v12){}; 

\draw[black, dotted, thick] (v1)  .. controls +(-6,0) and +(-6,0) .. (v2);
\draw[black, thick] (v1) -- (v3);
\draw[black, thick] (v2) -- (v4);
\draw[black, thick] (v3) -- (v4);
\draw[black, thick] (v3) -- (v5);
\draw[black, thick] (v4) -- (v5);
\draw[black, thick] (v5) -- (v6);
\draw[black, dotted, thick] (v6) -- (v7);
\draw[black, thick] (v7) -- (v8);
\draw[black, thick] (v8) -- (v9);
\draw[black, thick] (v8) -- (v10);
\draw[black, thick] (v9) -- (v10);
\draw[black, thick] (v9) -- (v11);
\draw[black, thick] (v10) -- (v12);
\draw[black, dotted, thick] (v11)  .. controls +(6,0) and +(6,0) .. (v12);

\node at (-5,0) {even};
\node at (4.5,2) {$\xleftarrow{\makebox[0.45cm]{}}$ odd $\xrightarrow{\makebox[0.45cm]{}}$};
\node at (14,0) {even};

\end{tikzpicture}
\hspace{1.5cm}
\begin{tikzpicture}[scale=0.26]

\node[inner sep=2.5pt, fill=black, circle] at (-4, 3.7)(v1){}; 
\node[inner sep=2.5pt, fill=black, circle] at (-4, -3.7)(v2){}; 
\node[inner sep=2.5pt, fill=black, circle] at (-2, 2)(v3){}; 
\node[inner sep=2.5pt, fill=black, circle] at (-2, -2)(v4){};

\node[inner sep=2.5pt, fill=black, circle] at (2, 2)(v9){}; 
\node[inner sep=2.5pt, fill=black, circle] at (2, -2)(v10){}; 
\node[inner sep=2.5pt, fill=black, circle] at (4, 3.7)(v11){}; 
\node[inner sep=2.5pt, fill=black, circle] at (4, -3.7)(v12){}; 

\draw[black, dotted, thick] (v1)  .. controls +(-6,0) and +(-6,0) .. (v2);
\draw[black, thick] (v1) -- (v3);
\draw[black, thick] (v2) -- (v4);
\draw[black, thick] (v3) -- (v4);
\draw[black, thick] (v3) -- (v9);
\draw[black, thick] (v3) -- (v10);
\draw[black, thick] (v4) -- (v9);
\draw[black, thick] (v4) -- (v10);
\draw[black, thick] (v9) -- (v10);
\draw[black, thick] (v9) -- (v11);
\draw[black, thick] (v10) -- (v12);
\draw[black, dotted, thick] (v11)  .. controls +(6,0) and +(6,0) .. (v12);

\node at (-5,0) {even};
\node at (5.2,0) {even};

\end{tikzpicture}
\end{center}
\vspace{-0.35cm}
\caption{Handcuffs and eye masks (the dotted lines represent odd paths)}
\label{fig:handcuff_eyemask}
\end{figure}
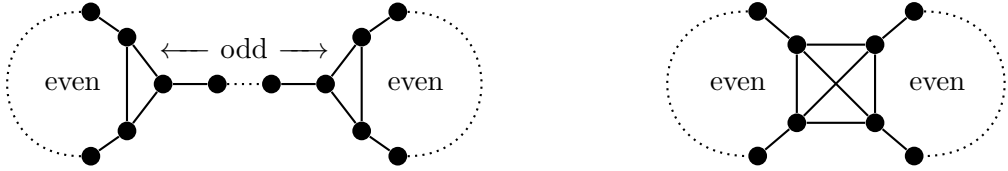

We say a graph is \emph{innocent} if it contains no odd holes, no antiholes of length at least six, no odd prisms, no handcuffs, and no eye masks. The following was proved by Wang in 2006.

\begin{theorem}\cite{Wang}
\label{thm:main}
Let $G$ be a claw-free graph. Then $G$ is strongly perfect if and only if $G$ is innocent.
\end{theorem}

The main result of this paper is a new shorter proof of Theorem \ref{thm:main}. It is easy to check that none of the forbidden induced subgraphs in an innocent graph has a strong stable set, and so they are not strongly perfect. Therefore, it is enough to prove the following:

\begin{theorem}
\label{thm:weaker}
Let $G$ be a claw-free innocent graph. Then $G$ has a strong stable set.
\end{theorem}

In this paper, we prove the following stronger variant of Theorem \ref{thm:weaker} (safe vertices are to be defined later.) The proof will be presented in Section \ref{sec:proof}.

\begin{theorem}
\label{thm:stronger}
Let $G$ be a claw-free innocent graph and let $Z$ be a consistent set of safe vertices in $G$. Then $G$ has a strong stable set that contains $Z$.
\end{theorem}

The \emph{line graph} $L(H)$ of a graph $H$ is the graph whose vertices are the edges of $H$ and whose edges are the pairs of edges of $H$ that share a vertex. Since line graphs are claw-free and do not contain antiholes of length at least seven, the following is a weaker version of Theorem \ref{thm:main}. It was proved by Ravindra in 1984 and it is immediately implied by Theorem \ref{thm:main}.

\begin{theorem}\cite{Ravindra}
\label{thm:line_graphs}
Let $G$ be a line graph. Then $G$ is strongly perfect if and only if $G$ contains no odd holes, no odd prisms, no handcuffs, and no eye masks.
\end{theorem}

Before we present the proof of Theorem \ref{thm:stronger}, we need a result concerning different types of decompositions in claw-free graphs. We now give the definitions of these decompositions. 

A \emph{clique cutset} of a graph $G$ is a clique $X$ such that $G\setminus X$ is not connected. Suppose that there is a partition $(V_1, V_2, X)$ of $V(G)$ such that $X$ is a clique, and $|V_1|, |V_2| \geq 2$, and $V_1$ is anticomplete to $V_2$. Then we say that $X$ is an \emph{internal clique cutset}. Two adjacent vertices of a graph $G$ are called \emph{twins} if (apart from each other) they have the same neighbors in $G$, and if there are two such vertices, we say ``$G$ admits twins". 

Let $A, B$ be disjoint subsets of $V(G)$. The pair $(A, B)$ is called a \emph{homogeneous pair} in $G$ if $A, B$ are cliques, and for every vertex $v \in V(G) \setminus (A \cup B)$, $v$ is either complete or anticomplete to $A$, and either complete or anticomplete to $B$. Let $(A,B)$ be a homogeneous pair, such that $A$ is neither complete nor anticomplete to $B$. (Thus, at least one of $A, B$ has at least two members.) In these circumstances we call $(A, B)$ a \emph{$W$-join}. A $W$-join $(A, B)$ is \emph{coherent} if the set of all vertices in $V(G) \setminus (A \cup B)$ that are complete to $(A \cup B)$ is a clique. A $W$-join $(A, B)$ is \emph{proper} if every vertex of $A$ is mixed on $B$, and every vertex of $B$ is mixed on $A$. We say that a homogeneous pair $(A, B)$ is \emph{square-connected} if for every partition of $A$ (resp. $B$) into nonempty sets $A'$ and $A''$ (resp. $B'$ and $B''$), there is a square in $A \cup B$ intersecting both $A'$ and $A''$ (resp. $B'$ and $B''$). It follows that if $(A, B)$ is square-connected, then every vertex of $A \cup B$ is in a square.

Suppose that $V_1, V_2$ is a partition of $V(G)$ such that $V_1, V_2$ are non-empty and $V_1$ is anticomplete to $V_2$. We call the pair $(V_1, V_2)$ a \emph{$0$-join} in $G$.
Next, suppose that $V_1, V_2$ is a partition $V(G)$, and for $i = 1, 2$ there is a subset $A_i \subseteq V_i$ such that:
\begin{itemize}
\itemsep0em
\item $A_1 \cup A_2$ is a clique, and for $i=1,2$, $A_i, V_i \setminus A_i$ are both non-empty,
\item $V_1 \setminus A_1$ is anticomplete to $V_2$ , and $V_2 \setminus A_2$ is anticomplete to $V_1$.
\end{itemize}
In these circumstances we call $(V_1, V_2)$ a \emph{$1$-join}. If also $|V_1|, |V_2| > 2$, then $(V_1, V_2)$ is a \emph{rich} $1$-join.

We say that a graph $G$ is a \emph{linear interval graph} if the vertices of $G$ can be numbered $v_1, \dots, v_n$ such that for all $i, j$ with $1 \leq i < j \leq n$, if $v_i$ is adjacent to $v_j$ then $\{v_i, v_{i+1}, \dots, v_{j}\}$ is a clique. Equivalently, a linear interval graph is a graph $G$ with a linear interval representation, which is a point on the real line for each vertex and a set of intervals, such that vertices $u$ and $v$ are adjacent in $G$ if and only if there is an interval containing both corresponding points on the real line. Observe that linear interval graphs do not contain holes. We need the following result:

\begin{theorem}\cite{Claw3}
\label{thm:claw_free_internal_cc_proper}
Let $G$ be a claw-free graph with an internal clique cutset such that $G$ does not admit twins, a $0$-join, or a $1$-join. Then every hole in $G$ has length four; if there is a $C_4$, then $G$ admits a coherent proper $W$-join, and otherwise $G$ is a linear interval graph.
\end{theorem}

We also need a result concerning claw-free perfect graphs. We continue with definitions from \cite{Maffray_Reed} in order to state this result. A \emph{cobipartite} graph $(X, Y)$ is the complement of a bipartite graph, where $X, Y$ is a partition of its vertex set into two cliques. A graph is called \emph{peculiar} if it can be obtained as follows: start with three, pairwise vertex-disjoint, cobipartite graphs $(A_1, B_2), (A_2, B_3), (A_3, B_1)$ such that none of them is a complete graph; add all edges between every two of them; then take three cliques $K_1, K_2, K_3$ that are pairwise disjoint and disjoint from the $A_i$'s and $B_i$'s; add all the edges between $K_i$ and $A_j \cup B_j$ for $j \neq i$; there is no other edge in the graph. A graph is called \emph{elementary} if its edges can be bicolored in such a way that every $P_3$ has its two edges colored differently.

An edge is \emph{flat} if it does not lie in a triangle. Let $xy$ be a flat edge of a graph $G=(V, E)$, and $B=(X,Y)$ be a connected cobipartite graph disjoint from $G$. We can build a new graph $G'$ obtained from $G \setminus \{x, y\}$ and $B$ by adding all possible edges between $X$ and $N(x) \setminus y$ and between $Y$ and $N(y) \setminus x$. We say that $G$ is \emph{augmented} along $xy$, that $xy$ is augmented, and $B$ will be called the {\em augment} of $xy$. When we augment with a cobipartite graph $(X, Y)$ a flat edge $xy$ in a graph $G$, it is easy to see that $X, Y$ is a homogeneous pair of the resulting graph. Moreover, since $xy$ is a flat edge, the vertices $x$ and $y$ have no common neighbor in $G$, and so in the resulting graph $G'$ the set of vertices in $V(G') \setminus (X \cup Y)$ that are complete to $(X \cup Y)$ is empty. This implies that $(X,Y)$ is a coherent homogeneous pair.

Now, let $x_1y_1, \dots, x_hy_h$ be $h$ pairwise disjoint flat edges of $G$. Let $(X_1,Y_1), \dots, (X_h, Y_h)$ be $h$ pairwise disjoint cobipartite graphs that are also disjoint from $G$. We can obtain a graph $G'$ by augmenting respectively each edge $x_iy_i$ with the augment $(X_i,Y_i)$. This graph is independent of the order of the augments. The graph $G'$ will be called an \emph{augmentation} of $G$.

In 1988, Chv\'atal and Sbihi proved a decomposition theorem for claw-free perfect graphs. They showed that claw-free perfect graphs either have a clique cutset or belong to two basic classes of graphs: elementary and peculiar graphs. As given above, the structure of peculiar graphs is determined precisely by their definition, but that is not the case for elementary graphs. Later, Maffray and Reed proved that an elementary graph is an augmentation of the line graph of a bipartite multigraph, giving a precise description of all elementary graphs.

\begin{theorem}\cite{Chvatal_Sbihi}
\label{thm:Chvatal_Sbihi}
A claw-free graph is perfect if and only if it is elementary, or is peculiar, or admits a clique cutset.
\end{theorem}

\begin{theorem}\cite{Maffray_Reed}
\label{thm:Maffray_Reed}
A graph is elementary if and only if it is an augmentation of the line graph of a bipartite multigraph.
\end{theorem}

We prove Theorem \ref{thm:main} (and in fact Theorem \ref{thm:stronger}) in Section \ref{sec:proof}, but let us sketch the proof here. We will also explain the step where the stronger form of Theorem \ref{thm:stronger}  is needed. Our approach is to apply known structure theorems to $G$, and then deal with each of the possible outcomes separately. While doing that, it is important to obtain more structural information than just ``$G$ has a simplicial vertex''. So our first step is to construct a graph $G'$ that is obtained from $G$ by deleting the set of all simplicial vertices (which turns out to be a consistent set of safe vertices). Being an induced subgraph of $G$, the graph $G'$ is claw-free and perfect. Therefore, By Theorem \ref{thm:Chvatal_Sbihi}, $G'$ either is peculiar, or is elementary, or admits a clique cutset. We deal with the cases where $G'$ is peculiar or $G'$ is elementary directly. If $G'$ admits a clique cutset, the fact that $G'$ is obtained from $G$ by deleting all simplicial vertices allows us to show that $G$ admits an internal clique cutset. Hence, by Theorem \ref{thm:claw_free_internal_cc_proper}, $G$ admits one of several highly structured decompositions. We then decompose $G$ and show that the partial solutions in the blocks of the decompositions can be combined to get the required strong stable set in $G$ (the step of combining the partial solutions is where the full strength of Theorem \ref{thm:stronger} is necessary).

\section{Some lemmas}\label{sec:lemmas}

In this section we state some observations and prove a few lemmas that we will use later.

\begin{observation}
\label{obs:simplicial}
If $v$ is a simplicial vertex in a graph $G$, and $S$ is a strong stable set of $G\setminus \{v\}$, then either $S$ or $S \cup \{v\}$ is a strong stable set of $G$.
\end{observation}

\begin{observation}
\label{obs:internal}
If $G$ is a graph with no simplicial vertices, then every clique cutset of $G$ is an internal clique cutset.
\end{observation}

\begin{observation}
\label{obs:parallel_edges}
Parallel edges in a multigraph $G$ correspond to twin vertices in the line graph of $G$.
\end{observation}

\begin{observation}
\label{obs:simplicial_clique}
If $G$ is a claw-free graph and $v \in V(G)$ is a simplicial vertex of $G$ with $N(v) = K$, then $K$ is a simplicial clique in $G\setminus v$.
\end{observation}

\begin{observation}
\label{obs:simplicial_stable}
If $G$ is a graph and no two simplicial vertices of $G$ are twins, then the simplicial vertices of $G$ form a stable set.
\end{observation}

\begin{proof}
Let $u, v$ be simplicial vertices of $G$ and assume that they are adjacent. Since $u$ and $v$ are not twins, we may assume that there is a vertex $w$ in $G$ adjacent to $u$ and non-adjacent to $v$. But this is a contradiction to the fact that $u$ is simplicial.
\end{proof}

\begin{observation}
\label{obs:simplicial_internal}
Let $G$ be a graph with no twins and $G'$ be the graph obtained from $G$ by deleting the simplicial vertices of $G$. Then, every clique cutset of $G'$ is an internal clique cutset of $G$.
\end{observation}

\begin{proof}
Let $S$ be the set of simplicial vertices of $G$. By Observation \ref{obs:internal} we may assume that $S \neq \emptyset$, and by Observation \ref{obs:simplicial_stable}, $S$ is stable. Let $K$ be a clique cutset in $G'$ where $(A_1, A_2, K)$ is a partition of $V(G')$ such that $A_1$ is anticomplete to $A_2$. Let $S_1 = \{s \in S : s \text{ has a neighbor in } A_1 \}$ and $S_2 = S \setminus S_1$. Since every $s \in S$ is simplicial, $S_1$ is anticomplete to $A_2$. Thus, $A_1 \cup S_1$ is anticomplete to $A_2 \cup S_2$, and so $K$ is a clique cutset in $G$. We claim that $K$ is internal. Suppose not, then for some $i \in \{1,2\}$, $|A_i \cup S_i| =1$. It follows that $S_i = \emptyset$ and $A_i = \{a\}$. But now $N_G(a) \subseteq K$, and so $a \in S$, a contradiction.
\end{proof}

We thank Sophie Spirkl for pointing out to us a short proof of the following lemma.

\begin{lemma}
\label{lem:hole_or_simplicial}
Let $K$ be a clique cutset in a graph $G$ where $(A, B, K)$ is a partition of $V(G)$ such that $A$ is anticomplete to $B$. Then, either there is a hole in $G | (K \cup B)$, or there is a simplicial vertex of $G$ in $B$.
\end{lemma}

\begin{proof}
We may assume that $G' = G | (K \cup B)$ is a chordal graph, for otherwise the conclusion of the theorem holds. If $G'$ is a complete graph, then every vertex of $B$ is a simplicial vertex of $G$. If $G'$ is not a complete graph, since it is chordal, there are two non-adjacent simplicial vertices $x, y$ in $G'$. Since $K$ is a clique, at least one of $x, y$ is in $B$, say $x \in B$. Then $x$ is a simplicial vertex of $G$ in $B$.
\end{proof}

A \emph{clown} is a graph with vertex set $\{c_0, c_1, \dots, c_k\}$ where $\{c_1, \dots, c_k\}$ is an even hole and $c_0$ is adjacent to exactly $c_1, c_2$, as in Figure \ref{fig:clown}. We call $c_0$ the \emph{hat} of the clown. We write $h(D)$ to denote the hat of a clown $D$.

\begin{figure}[h]
\begin{center}
\begin{tikzpicture}[scale=0.28]

\node[label=left:{$c_k$}, inner sep=2.5pt, fill=black, circle] at (0, 0)(v1){}; 
\node[label=right:{$c_3$}, inner sep=2.5pt, fill=black, circle] at (8, 0)(v2){}; 
\node[label=left:{$c_1$}, inner sep=2.5pt, fill=black, circle] at (2, 2)(v3){}; 
\node[label=right:{$c_2$}, inner sep=2.5pt, fill=black, circle] at (6, 2)(v4){};
\node[label=above:{$c_0$}, inner sep=2.5pt, fill=black, circle] at (4, 4.5)(v5){}; 

\draw[black, dotted, thick] (v1)  .. controls +(0,-6) and +(0,-6) .. (v2);
\draw[black, thick] (v1) -- (v3);
\draw[black, thick] (v2) -- (v4);
\draw[black, thick] (v3) -- (v4);
\draw[black, thick] (v3) -- (v5);
\draw[black, thick] (v4) -- (v5);

\node at (4,-0.8) {even};

\end{tikzpicture}
\end{center}
\vspace{-1cm}
\caption{A clown}
\label{fig:clown}
\end{figure}

A vertex $v$ in a graph $G$ is called \emph{safe} if it is simplicial and if for every clown $D$ in $G$ the following holds: every path $P$ from $v$ to $h(D)$ such that there are no edges between $P\setminus \{h(D)\}$ and $D \setminus \{h(D)\}$ has odd length. In particular, no safe vertex is a hat for a clown.

\begin{observation}
\label{obs:consistent_safe}
If $Z$ is a consistent set of safe vertices in a graph $G$, then $Z \setminus v$ is a consistent set of safe vertices in $G\setminus v$ for every $v \in V(G)$.
\end{observation}

\begin{observation}
\label{obs:path_to_hole}
Let $G$ be a claw-free innocent graph.
\begin{enumerate}
\itemsep0em
\item Let $p_1 \in V(G)$ and let $C$ be a hole in $G \setminus p_1$. Let $P = p_1 \d p_2 \d \dots \d p_k$ be a path in $G$ from $p_1$ to $C$. If $k \geq 2$ or $p_1$ is safe, then $C \cup \{p_k\}$ is a clown with hat $p_k$. Hence, if $p_1$ is safe, then $P$ is odd. 

\item In particular, for a safe vertex $v$ and a hole $H$ in $G$, $v \notin V(H)$ and $v$ has no neighbors in $H$.

\item Let $C_1, C_2$ be two disjoint holes in $G$ such that $V(C_1)$ is anticomplete to $V(C_2)$. Let $Q = q_1 \d q_2 \d \dots \d q_\ell$ be a path in $G$ with $\ell \geq 1$ such that it is a path from $q_1$ to $C_2$ and a path from $q_\ell$ to $C_1$. Then, $Q$ is even, and $C_1 \cup \{q_1\}$ is a clown with hat $q_1$, and $C_2 \cup \{q_\ell\}$ is a clown with hat $q_\ell$.
\end{enumerate}
\end{observation}

\begin{proof}
Let the vertices of $C$ be $c_1 \d \dots \d c_m \d c_1$ in order, and assume that $k \geq 2$ or $p_1$ is safe. If $p_k$ has only one neighbor in $C$, say $c_1$, then $\{c_1, c_2, c_m, p_k\}$ is a claw in $G$. Suppose $p_k$ has two non-adjacent neighbors $c_i, c_j$ in $C$. Since $k \geq 2$ or $p_1$ is safe (and so simplicial), it follows that $p_k \neq p_1$, and $\{c_i, c_j, p_k, p_{k-1}\}$ is a claw in $G$. Thus, $p_k$ has exactly two neighbors in $C$ and they are consecutive in $C$, and so $C \cup \{p_k\}$ is a clown in $G$. Hence, if $p_1$ is safe, then $P$ is odd. This proves the first assertion. For the second assertion, $v \notin V(H)$ since $v$ is simplicial, and $v$ has no neighbors in $H$ by the first assertion. Finally, the third assertion follows from the first assertion applied to the path-hole pairs $(Q, C_1)$ and $(Q, C_2)$, separately. Also, since $C_1 \cup V(Q) \cup C_2$ is not a handcuff, $Q$ has even length (possibly zero).
\end{proof}

The next lemma can be considered as the central observation of this paper. It lies in the heart of our results and it is this idea which makes the essence of our proof simple.

\begin{lemma}\label{lem:heart_of_it_all}
Let $G$ be a claw-free innocent graph and let $v_0$ be a simplicial vertex of $G$. Let
\begin{enumerate}
\itemsep0em
\item $G_1'$ be the graph obtained from $G$ by adding the vertices $\{v_1, v_2, \dots, v_m\}$, where $m \geq 3$ is odd, such that $v_0 \d v_1 \d \dots \d v_m \d v_0$ is a hole, $v_1$ is complete to $N_G[v_0]$ and anticomplete to $V(G) \setminus N_G[v_0]$, and $\{v_2, \dots, v_m\}$ is anticomplete to $V(G) \setminus \{v_0\}$.

\item $G_2'$ be the graph obtained from $G$ by adding the vertices $\{v_1, v_2, \dots, v_m\}$, where $m \geq 2$ is even, such that $v_0 \d v_1 \d \dots \d v_m$ is a path, and $\{v_1, v_2, \dots, v_m\}$ is anticomplete to $V(G) \setminus \{v_0\}$.

\item $G_3'$ be the graph obtained from $G$ by adding the vertices $\{v_1, v_2, \dots, v_m\}$, where $m \geq 2$ is even, and the vertices $\{c_0, c_1, \dots, c_k\}$, where $k \geq 4$ is even, such that $v_0 \d v_1 \d v_2 \d \dots \d v_m \d c_0$ is a path, $\{c_0, c_1, \dots, c_k\}$ is a clown with hat $c_0$, $\{v_1, v_2, \dots, v_m\}$ is anticomplete to $V(G) \setminus \{v_0\}$, and $\{c_1, \dots, c_k\}$ is anticomplete to $V(G) \cup \{v_1, v_2, \dots, v_m\}$.
\end{enumerate}
Then, for $i=1,2,3$, the following statements hold:
\begin{enumerate}[label=(\alph*)]
\itemsep0em
\item\label{it:first} If $v_0$ is a safe vertex of $G$, then $G_i'$ is claw-free and innocent. 
\item\label{it:second} Suppose $G_i'$ is claw-free and innocent. Let $Z'$ be a consistent set of safe vertices in $G_i'$ where if $i=2$ then $v_m \in Z'$. Then, $Z = (Z' \cap V(G)) \cup \{v_0\}$ is a consistent set of safe vertices in $G$. Also, if $S$ is a strong stable set in $G$ with $Z \subseteq S$, then $G_i'$ has a strong stable set $S'$ with $Z' \subseteq S'$.
\end{enumerate}
\end{lemma}

\begin{proof}
It is straightforward to check that \ref{it:first} holds. We prove \ref{it:second}. By Observation \ref{obs:consistent_safe}, it is true that $Z_G = Z' \cap V(G)$ is a consistent set of safe vertices in $G$. We want to show that $Z = Z_G \cup \{v_0\}$ is a consistent set of safe vertices in $G$. We claim that all paths from $z \in Z_G$ to $v_0$ in $G$ are even. Suppose there is an odd path $z \d p_1 \d \dots \d p_\ell \d v_0$ from $z \in Z_G$ to $v_0$ in $G$. Then, we get a contradiction in each case:

\vspace{-0.2cm}

\begin{itemize}
\itemsep-0.1em
\item $z$ is not safe in $G_1'$ since $\{p_\ell, v_0, v_1, \dots, v_m\}$ is a clown and the path $z \d p_1 \d \dots \d p_\ell$ is even;
\item $Z'$ is not consistent in $G_2'$ since the path $z \d p_1 \d \dots \d p_\ell \d v_0 \d v_1 \d \dots \d v_m$ is odd;
\item $z$ is not safe in $G_3'$ since $\{c_0, \dots, c_k\}$ is a clown and the path $z \d p_1 \d \dots \d p_\ell \d v_0 \d v_1 \d \dots \d v_m \d c_0$ is even.
\end{itemize}

\vspace{-0.2cm}

\noindent This proves that $Z$ is a consistent set in $G$. Next, we prove that $v_0$ is safe in $G$. Suppose not. Let $D$ be a clown in $G$ and let $P$ be an even path (possibly of length zero) in $G$ from $v_0$ to $h(D)$. Then, in $G_1'$ the set $D \cup P \cup \{v_1, v_2, \dots, v_m\}$ is either an eye mask or a handcuff; in $G_2'$ the vertex $v_m$ is not safe; in $G_3'$ the set $D \cup P \cup \{v_1, v_2, \dots, v_m\} \cup \{c_0, c_1, \dots, c_k\}$ is a handcuff. We get a contradiction in each case. Therefore, $Z$ is a consistent set of safe vertices in $G$. Next, let $S$ be a strong stable set in $G$ with $Z \subseteq S$. Note that $v_0 \in S$ since $v_0 \in Z$. Let $S_1' = S \cup \{v_2, v_4, \dots, v_{m-1}\}$, $S_2' = S \cup \{v_2, v_4, \dots, v_m\}$, $S_3' = S \cup \{v_2, v_4, \dots, v_m, c_2, c_4, \dots, c_k\}$. Then, $S_i'$ is a strong stable set in $G_i'$ with $Z' \subseteq S_i'$.
\end{proof}

Next, we use Lemma \ref{lem:heart_of_it_all} to prove that Theorem \ref{thm:weaker} and Theorem \ref{thm:stronger} are equivalent.

\begin{lemma}
\label{lem:weaker_stronger_same}
The following statements are equivalent:

\vspace{-0.15cm}

\begin{enumerate}
\itemsep-0.1em
\item \label{weaker} Every claw-free innocent graph has a strong stable set.
\item \label{stronger} Let $G$ be a claw-free innocent graph and let $Z$ be a consistent set of safe vertices in $G$. Then $G$ has a strong stable set that contains $Z$.
\end{enumerate}
\end{lemma}

\begin{proof}
Clearly, \ref{stronger} implies \ref{weaker}. Now, assume \ref{weaker} holds. Let $Z = \{z_1, \dots, z_m\}$ be a consistent set of safe vertices in $G$, and let $I = \{1,\dots,m\}$ denote the index set. Let $G'$ be the graph obtained from $G$ by adding three vertices $w_i, x_i, y_i$ for each $z_i$ such that $w_i$ is complete to $N_G[z_i] \cup \{x_i\}$, and $y_i$ is adjacent to $z_i$ and $x_i$. By successive applications of Lemma \ref{lem:heart_of_it_all}, $G'$ is claw-free and innocent. (Note here that after one application of Lemma \ref{lem:heart_of_it_all}, say for $z_1$, the set $Z \setminus \{z_1\}$ is a consistent set of safe vertices in the new graph since $Z$ is a consistent set in $G$.) Therefore, by \ref{weaker}, $G'$ has a strong stable set $S'$. For each $i \in I$, since $z_iy_i$, $y_ix_i$, and $x_iw_i$ are maximal cliques of $G'$, either $\{x_i, z_i\} \subseteq S'$ or $\{y_i, w_i\} \subseteq S'$. We observe that the set $S = S' \setminus \bigcup\limits_{i\in I} \{x_i, y_i\}$ is a strong stable set of $G' \setminus \bigcup\limits_{i\in I} \{x_i, y_i\}$. Since $z_i$ and $w_i$ are twins in $G' \setminus \bigcup\limits_{i\in I} \{x_i, y_i\}$, it follows that $G$ has a strong stable set containing $Z$.
\end{proof}

Next, we prove several lemmas about cobipartite graphs and peculiar graphs. We use the following well-known result several times.

\begin{theorem}\label{thm:chordal_bip}\cite{Golumbic_Ross}
Bipartite graphs with no hole of length at least six contain a simplicial edge.
\end{theorem}

\begin{lemma}
\label{lem:simplicial_edges}
Let $G$ be a cobipartite graph $(A, B)$ and let $u \in A$, $v \in B$ be such that $uv$ is a cosimplicial non-edge. Then $\{u, v\}$ is a strong stable set in $G$.
\end{lemma}

\begin{proof}
Suppose $C$ is a maximal clique in $G$ such that $C \cap \{u,v\} = \emptyset$. Since $C$ is maximal, there exist $c_u \in C \setminus N[u]$ and $c_v \in C \setminus N[v]$. Since $G$ is cobipartite, $c_u \neq c_v$. As $uv$ is a cosimplicial non-edge, $c_u, c_v$ are non-adjacent, contrary to the fact that $C$ is a clique.
\end{proof}

The following lemma provides some well-known properties of bipartite chain graphs \cite{HammerPS}. However, since the proof is simple, we include it for the sake of completeness.

\begin{lemma}
\label{lem:no_C4_order}
Let $G$ be a graph and let $A, B$ be two disjoint sets of vertices in $G$ such that for every $a, a' \in A$ and $b, b' \in B$, if $ab, a'b'$ are edges, then at least one of $ab', a'b$ is an edge. Then, one can order the vertices of $A$, say $a_1, \dots, a_p$, in such a way that if $i \leq j$, then $N(a_i) \cap B \subseteq N(a_j) \cap B$.
Moreover, either some vertex of $A$ is complete to $B$ or some vertex of $B$ is anticomplete to $A$. 
\end{lemma}

\begin{proof}
Let $u, v \in A$ and assume that $N(u) \cap B \not\subseteq N(v) \cap B$ and $N(v) \cap B \not\subseteq N(u) \cap B$. Then, there are vertices $x, y \in B$ such that $xu, yv$ are edges, and $xv, yu$ are non-edges, a contradiction. Next, assume that no vertex of $A$ is complete to $B$. Then, in particular, there is a vertex $b \in B$ non-adjacent to $a_p$. But since $N(a) \cap B \subseteq N(a_p) \cap B$ for every $a \in A$, it follows that $b$ is anticomplete to $A$.
\end{proof}

\begin{lemma}
\label{lem:cobipartite}
Let $G$ be an innocent cobipartite graph $(A, B)$. Let $Z$ be a consistent set of safe vertices in $G$. Then $G$ has a strong stable set that contains $Z$.
\end{lemma}

\begin{proof}
Since $G$ is cobipartite, it is claw-free. We may assume that $G$ is not a complete graph since otherwise $|Z| \leq 1$, and every vertex of $G$ is a strong stable set in $G$. Also, $|Z| \leq 2$ as $A$ and $B$ are cliques. By Lemma \ref{lem:simplicial_edges}, it is enough to show that there is a cosimplicial non-edge $ab$ in $G$ with $Z \subseteq \{a,b\}$. Note that since $G$ is innocent, $G$ contains no antihole of length at least six. So, if $Z = \emptyset$, then by Theorem \ref{thm:chordal_bip}, there is a cosimplicial non-edge $ab$ in $G$ with $Z \subseteq \{a,b\}$. Hence, we may assume that $Z \neq \emptyset$. Assume first that $|Z| = 2$. Let $Z \cap A = \{a\}$ and $Z \cap B = \{b\}$. Then, $ab$ is a cosimplicial non-edge, for otherwise there is a three-edge path in $G$ from $a$ to $b$, contrary to the fact that $Z$ is consistent. Next, assume that $|Z| = 1$. Without loss of generality, assume that $Z \cap A \neq \emptyset$ and let $Z \cap A = \{a\}$. Let $B_1, B_2$ be the sets of neighbors and the non-neighbors of $a$ in $B$, respectively. Then $B_1$ is complete to $A$ since $a$ is simplicial, and $B_2 \neq \emptyset$ as $G$ is not a complete graph. Moreover, by Observation \ref{obs:path_to_hole}, there is no $C_4$ in $G | (B_2 \cup A)$, and so by Lemma \ref{lem:no_C4_order}, we can order the vertices of $B_2$, say $b_1, b_2, \dots, b_p$, in such a way that if $i \leq j$, then $N(b_i) \cap A \subseteq N(b_j) \cap A$. Let $b \in B_2$ be such that $N(b') \cap A \subseteq N(b) \cap A$ for all $b' \in B_2$. Now, $ab$ is a cosimplicial non-edge in $G$, for otherwise $N(b_2) \cap A \not\subseteq N(b) \cap A$ for some $b_2 \in B_2$, a contradiction.
\end{proof}

The following follows from \cite{Chvatal1} and \cite{Chvatal2}, where it is proven that every perfectly orderable graph is strongly perfect, and that a peculiar graph is perfectly orderable, respectively. However, we prove a weaker statement and include its proof for the sake of completeness.

\begin{lemma}
\label{lem:peculiar}
Let $G$ be a peculiar graph containing no antihole of length at least six, then $G$ has a strong stable set.
\end{lemma}

\begin{proof}
We use the same notation as in the definition of a peculiar graph. Recall that $\tilde G = G | (A_1 \cup B_2)$ is a cobipartite graph. As $G$ contains no antihole of length at least six, by Theorem \ref{thm:chordal_bip}, $\tilde G$ contains a cosimplicial non-edge $uv$ with $u \in A_1$, $v \in B_2$. We claim that $\{u, v\}$ is a strong stable set in $G$. Suppose not and let $C$ be a maximal clique such that $C \cap \{u,v\} = \emptyset$. As $C$ is maximal, $u$ has non-neighbors in $C$, and so $C \cap (B_2 \cup K_1) \neq \emptyset$. Similarly, $C \cap (A_1 \cup K_2) \neq \emptyset$. Now, since $K_1$ is anticomplete to $A_1 \cup K_2$, it follows that $C \cap K_1 = \emptyset$. Similarly, $C \cap K_2 = \emptyset$. Therefore, $C \cap A_1 \neq \emptyset$ and  $C \cap B_2 \neq \emptyset$. Note that since $C \cap (K_1 \cup K_2) = \emptyset$ and $v$ is complete to $(B_1 \cup B_3 \cup A_2 \cup A_3 \cup K_3)$, it follows that $C \cap A_1$ contains a non-neighbor of $v$. Similarly $C \cap B_2$ contains a non-neighbor of $u$. Let $c_v \in C \cap A_1$ be a non-neighbor of $v$, and $c_u \in C \cap B_2$ be a non-neighbor of $u$. Then $c_v, c_u$ are non-adjacent since $uv$ is a cosimplicial non-edge in $\tilde G$, contrary to the fact that $C$ is a clique.
\end{proof}

\begin{lemma}
\label{lem:peculiar_no_simp_clique}
A peculiar graph $G$ does not have a simplicial clique.
\end{lemma}

\begin{proof}
We use the same notation as in the definition of a peculiar graph. Assume for a contradiction that $C$ is a simplicial clique in $G$. Suppose that $K_i \setminus C \neq \emptyset$ for every $i = 1,2,3$. Then, $C \cap A_i = \emptyset$ and $C \cap B_i = \emptyset$ for every $i = 1,2,3$, in which case we may assume that $C \subseteq K_1$ since $K_i$'s are anticomplete to each other. But $K_1$ is complete to $A_2 \cup B_3$ which is not a clique, a contradiction. Thus, we may assume that $K_1 \subseteq C$ (permuting the indices if necessary). So, $C \cap (K_2 \cup K_3) = \emptyset$ and $C \cap (A_1 \cup B_1) = \emptyset$ since $C$ is a clique.

Now, if $C \cap B_3 \neq \emptyset$, say $b_3 \in C \cap B_3$, then since $b_3$ is complete to $K_2 \cup B_2$, and $K_2$ is anticomplete to $B_2$, and $C \cap K_2$ is empty, it follows that $B_2 \subseteq C$. Similarly, if $C \cap B_2 \neq \emptyset$, then $B_3 \subseteq C$. Therefore, either $(B_2 \cup B_3) \subseteq C$ or $C \cap (B_2 \cup B_3) = \emptyset$. By symmetry, either $(A_2 \cup A_3) \subseteq C$ or $C \cap (A_2 \cup A_3) = \emptyset$. Now, $K_1 \subseteq C$, $K_1$ is complete to $A_2 \cup B_3$, and $A_2 \cup B_3$ is not a clique, hence we may assume that $(B_2 \cup B_3) \subseteq C$ and $C \cap (A_2 \cup A_3) = \emptyset$. But $B_2$ is complete to $B_1 \cup A_3$ which is not a clique, and $C \cap (B_1 \cup A_3) = \emptyset$, a contradiction.
\end{proof}

\section{Properties of a minimal counterexample}\label{sec:minimal_counter}

In this section we prove several statements about the structure of a minimal counterexample to Theorem \ref{thm:stronger}, which we restate here for ease of reference.

\begin{customtheorem}{2.3}
\label{thm:stronger2}
Let $G$ be a claw-free innocent graph and let $Z$ be a consistent set of safe vertices in $G$. Then $G$ has a strong stable set that contains $Z$.
\end{customtheorem}


We say that a pair $(G,Z)$ is a \emph{suspect} if $G$ is a claw-free innocent graph, $Z$ is a consistent set of safe vertices in $G$ such that no strong stable set in $G$ contains $Z$, and $G$ is chosen with $|V(G)|$ minimal subject to the existence of such a set $Z$.

\begin{theorem}\label{thm:suspect_starters}
If $(G,Z)$ is a suspect, then $G$ is not a complete graph, every simplicial vertex of $G$ is in $Z$, $G$ is connected (i.e. $G$ does not admit a $0$-join), and $G$ does not admit twins.
\end{theorem}

\begin{proof}
Let $(G,Z)$ be a suspect. Then, $G$ is not a complete graph since otherwise $|Z| \leq 1$ and every vertex of $G$ is a strong stable set in $G$.
Assume that $v$ is a simplicial vertex of $G$ and v $\notin Z$. By Observation \ref{obs:consistent_safe}, $Z$ is a consistent set of safe vertices in $G \setminus v$, and by the minimality of $G$, the graph $G \setminus v$ has a strong stable set $S$ that contains $Z$. By Observation \ref{obs:simplicial}, either $S$ or $S \cup \{v\}$ is a strong stable set in $G$ containing $Z$, a contradiction. Therefore, every simplicial vertex of $G$ is in $Z$.

Next, assume that $G$ is not connected. Let $G_1, \dots, G_k$ be the connected components of $G$ and let $Z_i = Z \cap V(G_i)$. Clearly, $Z_i$ is a consistent set of safe vertices in $G_i$. By the minimality of $G$, each $G_i$ has a strong stable set $S_i$ containing $Z_i$. Then, $S = S_1 \cup \dots \cup S_k$ is a strong stable set in $G$ containing $Z$, a contradiction. Hence, $G$ is connected.

Finally, let $a$ and $b$ be twins in $G$. Since they are adjacent, at least one of them is not in $Z$, say $a \notin Z$. Let $G' = G \setminus a$. By Observation \ref{obs:consistent_safe}, $Z$ is a consistent set of safe vertices in $G'$. By the minimality of $G$, $G'$ has a strong stable set $S$ containing $Z$. Suppose that $S$ is not a strong stable set of $G$. Then there is a maximal clique $K$ in $G$ not meeting $S$. It follows that $K \setminus \{a\}$ is not a maximal clique of $G'$, and so there is a vertex $v \notin K$ in $G'$ that is complete to $K \setminus \{a\}$. Since $K \cup \{v\}$ is not a clique of $G$, we deduce that $a \in K$ and $v$ is non-adjacent to $a$. But by the maximality of $K$, $b \in K$, and so $v$ is adjacent to $b$, contrary to the fact that $a$ and $b$ are twins.
\end{proof}

\begin{theorem}
\label{thm:not_linear_interval}
If $(G,Z)$ is a suspect, then $G$ is not a linear interval graph.
\end{theorem}

\begin{proof}
Suppose $G$ is a linear interval graph. Let $v_1, \dots, v_k$ be the vertices of $G$ in order. By Theorem \ref{thm:suspect_starters}, $v_1, v_k$ are in $Z$ since they are simplicial. Let $i$ be maximum such that $v_2$ is adjacent to $v_i$. Note that $i \neq 1$ as $G$ is connected and not a complete graph by Theorem \ref{thm:suspect_starters}. Also, since $v_1, v_2$ are not twins, $v_1$ is not adjacent to $v_i$; and since $G$ is connected, $v_1$ is adjacent to $v_2$. Note that no neighbor $v_m$ of $v_i$, where $m > i$, is in $Z$ since otherwise $v_1 \d v_2 \d v_i \d v_m$ is a three-edge path, contrary to $v_1, v_m \in Z$. 

Moreover, we claim that $v_t \notin Z$ for $2 \leq t \leq i-1$. Assume that $v_j \in Z$ where $2 \leq j \leq i-1$. Then, 
\begin{itemize}
\itemsep-0.2em
\item $v_j$ is not adjacent to $v_1$ because $Z$ is stable, 
\item $v_j$ is complete to $\{v_2, v_i\}$ since $v_2$ is adjacent to $v_i$ and  $2 \leq j \leq i-1$,
\item $v_j$ is not adjacent to $v_\ell$ for $\ell > i$ since $v_j$ is simplicial and $v_j$ is adjacent to $v_2$.
\end{itemize}

Let $P$ be a path from $v_i$ to $v_k$ in $G \setminus \{v_1, \dots, v_{i-1}\}$. Then one of $v_1 \d v_2 \d v_i \d P \d v_k$ and $v_j \d v_i \d P \d v_k$ is an odd path in $G$ between two members of $Z$, contrary to the fact that $Z$ is consistent. This proves that $v_t \notin Z$ for $2 \leq t \leq i-1$.

Now, let $\tilde G = G \setminus \{v_1, \dots, v_{i-1}\}$ and let $\tilde Z = (Z \setminus \{v_1\}) \cup \{v_i\}$. Then, by Lemma \ref{lem:heart_of_it_all} applied to the graph $G \setminus \{v_3, \dots, v_{i-1}\}$, $\tilde Z$ is a consistent set of safe vertices in $\tilde G$. By the minimality of $G$, there is a strong stable set $\tilde S$ in $\tilde G$ containing $\tilde Z$, and by Lemma \ref{lem:heart_of_it_all}, $S= \tilde S \cup \{v_1\}$ is a strong stable set in $G$ containing $Z$.
\end{proof}

\begin{theorem}
\label{thm:no_prop_coh_W}
If $(G,Z)$ is a suspect then $G$ does not admit a proper coherent $W$-join.
\end{theorem}

\begin{proof}
Assume that $G$ admits a proper coherent $W$-join. Let $V(G) = A \cup B \cup C \cup D \cup E \cup F$ be a partition of $V(G)$ where $(A,B)$ is a homogeneous pair, $C$ is complete to $A$ and anticomplete to $B$, $D$ is complete to $B$ and anticomplete to $A$, $E$ is complete to $A\cup B$, and $F$ is anticomplete to $A \cup B$. As this is a proper coherent $W$-join, $E$ is a clique (possibly empty), no vertex of $A$ is complete or anticomplete to $B$, no vertex of $B$ is complete or anticomplete to $A$. Then, by Lemma \ref{lem:no_C4_order}, $G | (A\cup B$) contains a $C_4$. It follows from Observation \ref{obs:path_to_hole} that $Z \subseteq F$.

Also, $C$ is a clique, for otherwise $G$ has a claw with center in $A$, two leaves in $C$ and the third leaf in $B$, a contradiction since $G$ is claw-free. Similarly, $D$ is a clique. Moreover, $E$ is anticomplete to $F$, for otherwise there is a claw in $G$ with center in $E$, two leaves in $A \cup B$, and one leaf in $F$.

Let $H$ be a $C_4$ in $G | (A\cup B$). If there is a path $P$ from $C$ to $D$ with $P^* \subseteq F$, then $G | (V(H) \cup V(P))$ is either an odd prism or contains an odd hole, contrary to the fact that $G$ is innocent. It follows that $F = F_C \cup F_D$ such that $C \cup F_C$ is anticomplete to $D \cup F_D$.

Now, let $G_C = G | (C \cup F_C)$ and $G_D = G | (D \cup F_D)$. Let $Z_C =  Z \cap F_C$ and $Z_D =  Z \cap F_D$. Let $G_1$ be obtained from $G_C$ by adding a new vertex $v_1$ complete to $C$. Similarly, let $G_2$ be obtained from $G_D$ by adding a new vertex $v_2$ complete to $D$. Let $Z_1 = Z_C \cup \{v_1\}$ and $Z_2 = Z_D \cup \{v_2\}$. By Lemma \ref{lem:heart_of_it_all}, for $i =1,2$, $Z_i$ is a consistent set of safe vertices in $G_i$. This is because the graph $G | (H \cup C \cup F_C)$ is an induced subgraph of $G$, hence claw-free innocent, and $G_1$ is obtained by replacing $H$ with $v_1$.

By the minimality of $G$, there exists a strong stable set $S_i$ in $G_i$ containing $Z_i$. Let $S_C = S_1 \setminus \{v_1\}$ and $S_D = S_2 \setminus \{v_2\}$. Note that $Z \subseteq S_C \cup S_D$. Let $G_{AB} = G | (A \cup B)$. Since $G$ contains no antihole of length at least six, by Theorem \ref{thm:chordal_bip}, $G_{AB}$ contains a cosimplicial non-edge $ab$ with $a \in A$, $b \in B$. Now, $S = S_C \cup S_D \cup \{a, b\}$ is a strong stable set in $G$ containing $Z$. This is because if $K$ is a maximal clique of $G$, then either $K \subseteq C \cup F_C$ and $K \cap S_C \neq \emptyset$, or $K \subseteq D \cup F_D$ and $K \cap S_D \neq \emptyset$, or $K \subseteq A \cup C \cup E$ and $a \in K$, or $K \subseteq B \cup D \cup E$ and $b \in K$, or $K \subseteq A \cup B \cup E$ and $K \cap \{a, b\} \neq \emptyset$.
\end{proof}

\begin{lemma}
\label{lem:properties_1_join}
Let $(G,Z)$ be a suspect. Assume that $G$ admits a $1$-join $(V_1, V_2)$ and let $A_1, A_2$ be as in the definition of a $1$-join. Then, $Z \cap (A_1 \cup A_2) = \emptyset$.
\end{lemma}

\begin{proof}
For $i = 1, 2$, let $Z_i = V_i \cap Z$ and $G_i = G | V_i$. Notice that $| Z \cap (A_1 \cup A_2) | \leq 1$ since $A_1 \cup A_2$ is a clique. We may assume that $Z \cap A_2 \neq \emptyset$, say $Z \cap A_2 = \{a_2\}$. Now, let $G_1' = G | (V_1 \cup \{a_2\})$. Since $G_1'$ is an induced subgraph of $G$, by Observation \ref{obs:consistent_safe}, $Z_1' = Z_1 \cup \{a_2\}$ is a consistent set of safe vertices in $G_1'$. Observe that $|V(G_1')| < |V(G)|$ and $|V(G_2)| < |V(G)|$. By the minimality of $G$, there exist a strong stable set $S_1'$ in $G_1'$ containing $Z_1'$, and a strong stable set $S_2$ in $G_2$ containing $Z_2$. Now, $S_1' \cup S_2$ is a strong stable set in $G$ containing $Z$. This is because if $K$ is a maximal clique of $G$, then either $K \subseteq V_1$ and $K \cap S_1' \neq \emptyset$, or $K \subseteq V_2$ and $K \cap S_2 \neq \emptyset$, or $K = A_1 \cup A_2$ and $a_2 \in K$.
\end{proof}

\begin{theorem}
\label{thm:no_rich_1_join}
If $(G,Z)$ is a suspect, then $G$ does not admit a rich $1$-join.
\end{theorem}

\begin{proof}
Assume that $G$ admits a rich $1$-join $(V_1, V_2)$ and let $A_1, A_2$ be as in the definition of a $1$-join. For $i = 1, 2$, let $B_i = V_i \setminus A_i$ and $Z_i = V_i \cap Z$, and let $G_i = G | V_i$. Then by Lemma \ref{lem:properties_1_join}, $Z \cap (A_1 \cup A_2) = \emptyset$, and so $Z_i = B_i \cap Z$. Also, since $A_i$ is a clique cutset of $G$, by Lemma \ref{lem:hole_or_simplicial}, either there is an even hole in $G_i$ or there is a simplicial vertex of $G$ in $B_i$. If there is an even hole $H$ in $G_i$, then $|H \cap A_i|$ is either zero or two: $|H \cap A_i| \leq 2$ as $A_i$ is a clique, and if $|H \cap A_i| = 1$, then $G$ contains a claw. Now, let $O_i$ ($O$ for object) denote an even hole in $G_i$ or a simplicial vertex of $G$ in $B_i$. Let $a_1 \in A_1$ and $a_2 \in A_2$. Let $Q_1$ be a path from $a_2$ to $O_1$ and $Q_2$ be a path from $a_1$ to $O_2$. If $O_i$ is an even hole, then by Observation \ref{obs:path_to_hole}, the last vertex of $Q_i$ together with $O_i$ form a clown. Now, if $O_i$ is a simplicial vertex, let $P_i = Q_i$, and if $O_i$ is an even hole, let $P_i = Q_i \cup \{v_i\}$ where $v_i \in O_i$ is a neighbor of the last vertex of $Q_i$.

\begin{itemize}
\itemsep0em
\item If, for $i =1, 2$, $O_i$ is a simplicial vertex of $G$ in $B_i$, then $O_i \in Z_i$ by Theorem \ref{thm:suspect_starters}. Then, $P_1, P_2$ have different parities since otherwise $P_1 \cup P_2 \setminus \{a_1, a_2\}$ is an odd path from $O_1$ to $O_2$, contrary to $O_1, O_2 \in Z$.

\item If, for $i =1, 2$, $O_i$ is an even hole in $G_i$, then we may assume that $O_1 \subseteq B_1$, for otherwise if $O_i \cap A_i = 2$ for $i =1, 2$, then $O_1 \cup O_2$ is an eye mask in $G$. Now, $P_1, P_2$ have different parities since otherwise $O_1 \cup O_2 \cup P_1 \cup P_2 \setminus \{a_2, a_2\}$ is a handcuff in $G$.

\item If $O_1$ is a simplicial vertex of $G$ in $B_1$ and $O_2$ is an even hole in $G_2$, then $O_1 \in Z$ and $P_1, P_2$ have different parities since otherwise $Q_1 \cup Q_2 \setminus \{a_1, a_2\}$ is an even path from $O_1$ to $O_2$, contrary to Observation \ref{obs:path_to_hole}.
\end{itemize}

Therefore, since $a_1, a_2$ were chosen arbitrarily, we may assume that all paths like $P_1$ are even, and all paths like $P_2$ are odd. Let $G_1'$ be the graph obtained from $G_1$ by adding a new vertex $v_1$ complete to $A_1$. Let $G_2'$ be the graph obtained from $G_2$ by adding two new vertices $v_2, v_3$ and making $v_3$ complete to $A_2$ and $v_2$ adjacent to only $v_3$. Let $Z_i' = Z_i \cup \{v_i\}$ for $i = 1, 2$. Then, by Lemma \ref{lem:heart_of_it_all}, $Z_i'$ is a consistent set of safe vertices in $G_i'$. Next, observe that $|V(G_1')| < |V(G)|$ and $|V(G_2')| < |V(G)|$ as $(V_1, V_2)$ is a rich $1$-join. Then, by the minimality of $G$, there exist a strong stable set $S_1'$ in $G_1'$ containing $Z_1'$ and a strong stable set $S_2'$ in $G_2'$ containing $Z_2'$. Notice that $S_1' \cap A_1 = \emptyset$ since $v_1 \in S_1'$, and $S_2' \cap A_2 \neq \emptyset$ since $A_2 \cup \{v_3\}$ is a maximal clique in $G_2'$ and $v_2 \in S_2'$ (and so $v_3 \notin S_2'$). Let $S_i = S_i' \setminus \{v_i\}$. Now, $S_1 \cup S_2$ is a strong stable set in $G$ containing $Z$. This is because if $K$ is a maximal clique of $G$, then either $K \subseteq V_1$ and $K \cap S_1 \neq \emptyset$, or $K \subseteq V_2$ and $K \cap S_2 \neq \emptyset$, or $K = A_1 \cup A_2$ and $A_2 \cap S_2 \neq \emptyset$.
\end{proof}

\begin{lemma}
Let $A, B$ be two disjoint cliques in a graph $G$ such that there is a square in $G | (A \cup B)$. Assume that if $a_1 \d b_1 \d b_2 \d a_2 \d a_1$ is a square in $G | (A \cup B)$ with $a_1, a_2 \in A$ and $b_1, b_2 \in B$, then no vertex of $V(G) \setminus (A \cup B)$ is mixed on $\{a_1, a_2\}$, and no vertex of $V(G) \setminus (A \cup B)$ is mixed on $\{b_1, b_2\}$. Then, $G$ admits a proper coherent $W$-join.
\label{lem:get_W_join}
\end{lemma}

\begin{proof}
Let $H = a_1 \d b_1 \d b_2 \d a_2 \d a_1$ be a square in $G | (A \cup B)$ with $a_1 ,a_2 \in A$ and $b_1, b_2 \in B$. Consider an inclusion-wise maximal square-connected pair of cliques $(S, T)$ with $\{a_1, a_2\}\subseteq S \subseteq A$ and $\{b_1, b_2\} \subseteq T \subseteq B$. We claim that $(S, T)$ is a proper coherent $W$-join in $G$.

Since $G | (S \cup T$) contains a square, $S$ is neither complete nor anticomplete to $T$, and both $S$ and $T$ has at least two members. We show that for every vertex $v \in V(G) \setminus (S \cup T)$, $v$ is either complete or anticomplete to $S$, and either complete or anticomplete to $T$. Suppose $v \in V(G) \setminus (S \cup T)$ is mixed on $S$. Since $(S, T)$ is square-connected, there exists a square $s_1 \d s_2 \d t_2 \d t_1 \d s_1$ with $s_1, s_2 \in S$, $t_1, t_2 \in T$, such that $v$ is adjacent to $s_1$ and non-adjacent to $s_2$ (here, we apply the axiom of square-connectedness to $S \cap N(v)$ and $S \setminus N(v)$). Then $v \notin A$ as $A$ is a clique, and $v \notin V(G) \setminus (A \cup B)$ by the assumption of the lemma. Hence, $v \in B$. Then, since $B$ is a clique, $v$ is complete to $T$ and $v \d s_1 \d s_2 \d t_2 \d v$ is a square. It is now straightforward to check that $(S, T \cup \{v\})$ is a square-connected pair of cliques, contrary to the maximality of $(S, T)$. Therefore, no vertex in $V(G) \setminus (S \cup T)$ is mixed on $S$. The argument is symmetric if a vertex $v \in V(G) \setminus (S \cup T)$ is mixed on $T$. Hence, no vertex in $V(G) \setminus (S \cup T)$ is mixed on $T$ either. This proves that $(S, T)$ is a $W$-join.

Since $(S, T)$ is square-connected, every vertex of $S \cup T$ is in a square. Therefore, no vertex of $S$ is complete or anticomplete to $T$, and no vertex of $T$ is complete or anticomplete to $S$. Hence, $(S, T)$ is a proper $W$-join. Next, let $E$ be the set of all vertices in $V(G) \setminus (S \cup T)$ that are complete to $(S \cup T)$. We show that $E$ is a clique. Let $e_1, e_2 \in E$ be non-adjacent. Then, $S \cup \{e_1\}$ and $T \cup \{e_2\}$ are cliques, and $\{e_1, a_2, e_2, b_1\}$ is a square. So, $(S \cup \{e_1\}, T \cup \{e_2\})$ is square-connected, a contradiction to the maximality of $(S, T)$. Hence, $E$ is a clique, and $(S, T)$ is a proper coherent $W$-join.
\end{proof}

\begin{theorem}
\label{thm:no_1_join}
If $(G,Z)$ is a suspect, then $G$ does not admit a $1$-join.
\end{theorem}

\begin{proof}
Assume that $G$ admits a $1$-join $(V_1, V_2)$ and let $A_1, A_2$ be as in the definition of a $1$-join. For $i = 1, 2$, let $B_i = V_i \setminus A_i$ and $Z_i = V_i \cap Z$, and let $G_i = G | V_i$. Then, by Lemma \ref{lem:properties_1_join}, $Z \cap (A_1 \cup A_2) = \emptyset$, and so $Z_i = B_i \cap Z$. Moreover, by Theorem \ref{thm:no_rich_1_join}, $(V_1, V_2)$ is not a rich $1$-join. Therefore, we may assume that $|B_1| = |A_1| = 1$. Let $B_1 = \{b_1\}$ and $A_1 = \{a_1\}$. Then by Theorem \ref{thm:suspect_starters}, $b_1 \in Z$ since it is simplicial in $G$. Note that every vertex in $A_2$ has a neighbor in $B_2$ since no vertex of $A_2$ is simplicial because $Z \cap A_2 = \emptyset$. Also, $A_2$ is a simplicial clique in $G_2$, for otherwise if $b_2, b_2' \in B_2$ are two non-adjacent neighbors of some vertex $a_2$ in $A_2$, then $\{a_1, a_2, b_2, b_2'\}$ is a claw.

Let $N_1$ be the set of vertices in $B_2$ with a neighbor in $A_2$. Let $N_1'$ be the set of vertices in $N_1$ that are complete to $A_2$ and let $N_1'' = N_1 \setminus N_1'$. By definition, every vertex of $N_1''$ has a neighbor and a non-neighbor in $A_2$.

\vspace{-0.6cm}

\begin{equation}\label{eqn:where_is_Z}
\longbox{{\it $N_1 \cap Z = \emptyset$. Thus, $Z \subseteq (B_2 \setminus N_1) \cup \{b_1\}$.}} \tag{$4.7.1$}
\end{equation}

\begin{proof}
This follows from the fact that there is a three-edge path from $b_1$ to every vertex of $N_1$.
\end{proof}

\vspace{-0.8cm}

\begin{equation}\label{eqn:A2_not_maximal}
\longbox{{\it $A_2$ is not a maximal clique of $G_2$, so $N_1' \neq \emptyset$. Also, $N_1'$ is a clique and complete to $N_1''$.}} \tag{$4.7.2$}
\end{equation}

\begin{proof}
If $A_2$ is a maximal clique of $G_2$, then by the minimality of $G$, there exists a strong stable set $S_2$ in $G_2$ containing $Z_2$, and $S_2 \cup \{b_1\}$ is a strong stable set in $G$ containing $Z$. Therefore, $A_2$ is not a maximal clique of $G_2$. Then, some vertex of $B_2$ is complete to $A_2$, and so $N_1' \neq \emptyset$. Also, $N_1'$ is a clique and it is complete to $N_1''$ since $A_2$ is a simplicial clique in $G_2$.
\end{proof}

\vspace{-0.8cm}

\begin{equation}\label{eqn:n2_complete}
\longbox{{\it If $n_2 \in B_2 \setminus N_1$ is adjacent to a vertex $n_1'$ of $N_1'$, then $n_2$ is complete to $N_1''$.}} \tag{$4.7.3$}
\end{equation}

\begin{proof}
Let $n_1''$ be a non-neighbor of $n_2$ in $N_1''$ and $a_2$ be a non-neighbor of $n_1''$ in $A_2$. Then, $\{n_1', a_2, n_1'', n_2\}$ is a claw, a contradiction.
\end{proof}

\vspace{-0.8cm}

\begin{equation}\label{eqn:attachments_of_C}
\longbox{{\it For every component $C$ of $B_2 \setminus N_1$, $N(C) \cap N_1$ is a clique.}} \tag{$4.7.4$}
\end{equation}

\begin{proof}
Let $n, n' \in N(C) \cap N_1$ and assume that they are non-adjacent. Then $n, n' \in N_1''$. Since $A_2$ is a simplicial clique in $G_2$, no vertex of $A_2$ is complete to $\{n, n'\}$, and so there exist $a, a' \in A_2$ such that $n \d a \d a' \d n'$ is a path. Let $P$ be a path from $n$ to $n'$ with $P^* \subseteq C$. Since $V(P) \cup \{a, a'\}$ is a hole, it follows that $P$ is odd. Now, let $n_1'$ be a vertex in $N_1'$. Since $V(P) \cup \{n_1'\}$ is not an odd hole, $n_1'$ has a neighbor $p$ in $P^*$. Since $P$ is odd, we may assume that $p$ is not adjacent to $n$, contrary to \eqref{eqn:n2_complete}.
\end{proof}

By \eqref{eqn:where_is_Z}, $Z \subseteq (B_2 \setminus N_1) \cup \{b_1\}$. Let $a_2 \in A_2$ and let $G_2' = G | (B_2 \cup \{a_2\})$. Then, by Lemma \ref{lem:heart_of_it_all} applied to $G | (\{b_1, a_1, a_2\} \cup B_2)$, $Z_2' = Z_2 \cup \{a_2\}$ is a consistent set of safe vertices in $G_2'$. Let $S$ be a strong stable set in $G_2'$ containing $Z_2'$. Since $S \cup \{b_1\}$ is not a strong stable set in $G$, there exists a maximal clique $C$ of $G$ such that $C \cap (S \cup \{b_1\}) = \emptyset$. Since $C$ is not a maximal clique of $G_2'$, and since $C \neq \{a_1, b_1\}$, we deduce that $C$ meets $A_2$. Then, $C \subseteq \{a_1\} \cup A_2 \cup N_1$. But since $a_2 \notin C$, $C$ contains a non-neighbor of $a_2$, so $C \cap N_1 \neq \emptyset$; consequently $a_1 \notin C$. It follows that $N_1' \subseteq C \subseteq A_2 \cup N_1$. Also, since $a_2$ was chosen arbitrarily, we deduce that no vertex of $A_2$ is complete to $N_1$. Next, observe that $C \cap N_1$ is not a maximal clique of $G_2'$, for otherwise $S$ meets $C \cap N_1$, and in particular, $C \cap S \neq \emptyset$, a contradiction. Thus, there exists a vertex $n_2 \in B_2 \setminus N_1$ such that $n_2$ is complete to $C \cap N_1$. In particular, $n_2$ is complete to $N_1'$, and by \eqref{eqn:n2_complete}, $n_2$ is complete to $N_1$. Then, by \eqref{eqn:attachments_of_C}, $N_1$ is a clique.

Therefore, $G | (A_2 \cup N_1'')$ is a cobipartite graph. Since no vertex of $N_1''$ is anticomplete to $A_2$, and no vertex of $A_2$ is complete to $N_1''$, by Lemma \ref{lem:no_C4_order}, we deduce that there is a square in $G | (A_2 \cup N_1'')$. Let $H = p \d q \d q' \d p' \d p$ be a square in $G | (A_2 \cup N_1'')$ with $p ,q \in A_2$ and $p', q' \in N_1''$. Notice that no vertex of $V(G) \setminus (A_2 \cup N_1'')$ is mixed on $\{p, q\}$. Assume that a vertex $v \in V(G) \setminus (A_2 \cup N_1'')$ is mixed on $\{p', q'\}$, say $v$ is adjacent to $p'$ and non-adjacent to $q'$. Then, $v \in B_2 \setminus N_1$, and so $v$ is anticomplete to $\{p, q\}$. Then $\{p', v, p, q'\}$ is a claw, a contradiction. Thus, no vertex of $V(G) \setminus (A_2 \cup N_1'')$ is mixed on $\{p', q'\}$. Now, by Lemma \ref{lem:get_W_join} applied to the pair of cliques $(A_2, N_1'')$, $G$ admits a proper coherent $W$-join, a contradiction to Theorem \ref{thm:no_prop_coh_W}.
\end{proof}

Next, we prove a theorem about the line graphs of bipartite multigraphs. It is straightforward that if a graph $G = L(H)$ does not contain $L(F)$ for some graph $F$ as an induced subgraph, then $H$ does not contain $F$ as a subgraph (not necessarily induced). Observe that handcuffs and eye masks are the line graphs of what we call ``bicycles"; and that odd prisms are the line graphs of ``thetas". A \emph{bicycle} is a graph formed by connecting two vertex-disjoint even cycles $C_1, C_2$ via an even path from a vertex of $C_1$ to a vertex of $C_2$ where the length of the path is allowed to be zero. A \emph{theta} is a graph consisting of two non-adjacent vertices $a, b$ and three even paths $P_1, P_2, P_3$ of length at least two, each joining $a, b$ and otherwise vertex-disjoint (see Figure \ref{fig:bicycle_theta}, where the dotted lines represent even paths). Let us call a bipartite graph \emph{harmless} if its line graph is innocent. Then, harmless graphs do not contain bicycles or thetas as a subgraph (not necessarily induced).

\vspace{-0.1cm}

\begin{figure}[h]
\begin{center}
\begin{tikzpicture}[scale=0.22]

\node[inner sep=2.5pt, fill=black, circle] at (-4, 4)(v1){}; 
\node[inner sep=2.5pt, fill=black, circle] at (-4, -4)(v2){}; 
\node[inner sep=2.5pt, fill=black, circle] at (-2, 2.6)(v3){}; 
\node[inner sep=2.5pt, fill=black, circle] at (-2, -2.6)(v4){};

\node[inner sep=2.5pt, fill=black, circle] at (-0.7, 0)(v5){}; 
\node[inner sep=2.5pt, fill=black, circle] at (3, 0)(v6){}; 
\node[inner sep=2.5pt, fill=black, circle] at (6, 0)(v7){}; 
\node[inner sep=2.5pt, fill=black, circle] at (9.7, 0)(v8){};

\node[inner sep=2.5pt, fill=black, circle] at (11, 2.6)(v9){}; 
\node[inner sep=2.5pt, fill=black, circle] at (11, -2.6)(v10){}; 
\node[inner sep=2.5pt, fill=black, circle] at (13, 4)(v11){}; 
\node[inner sep=2.5pt, fill=black, circle] at (13, -4)(v12){}; 

\draw[black, dotted, thick] (v1)  .. controls +(-6,0) and +(-6,0) .. (v2);
\draw[black, thick] (v1) -- (v3);
\draw[black, thick] (v2) -- (v4);
\draw[black, thick] (v3) -- (v5);
\draw[black, thick] (v4) -- (v5);
\draw[black, thick] (v5) -- (v6);
\draw[black, dotted, thick] (v6) -- (v7);
\draw[black, thick] (v7) -- (v8);
\draw[black, thick] (v8) -- (v9);
\draw[black, thick] (v8) -- (v10);
\draw[black, thick] (v9) -- (v11);
\draw[black, thick] (v10) -- (v12);
\draw[black, dotted, thick] (v11)  .. controls +(6,0) and +(6,0) .. (v12);

\node at (-5,0) {even};
\node at (4.5,1.9) {$\xleftarrow{\makebox[0.43cm]{}}$ even $\xrightarrow{\makebox[0.43cm]{}}$};
\node at (14,0) {even};

\end{tikzpicture}
\hspace{1.5cm}
\begin{tikzpicture}[scale=0.27]

\node[inner sep=2.5pt, fill=black, circle] at (2, 2.6)(v3){}; 
\node[inner sep=2.5pt, fill=black, circle] at (2, -2.6)(v4){};

\node[inner sep=2.5pt, fill=black, circle] at (-0.7, 0)(v5){}; 
\node[inner sep=2.5pt, fill=black, circle] at (3, 0)(v6){}; 
\node[inner sep=2.5pt, fill=black, circle] at (9, 0)(v7){}; 
\node[inner sep=2.5pt, fill=black, circle] at (12.7, 0)(v8){};

\node[inner sep=2.5pt, fill=black, circle] at (10, 2.6)(v9){}; 
\node[inner sep=2.5pt, fill=black, circle] at (10, -2.6)(v10){}; 

\draw[black, thick] (v3) -- (v5);
\draw[black, thick] (v4) -- (v5);
\draw[black, thick] (v5) -- (v6);
\draw[black, dotted, thick] (v6) -- (v7);
\draw[black, dotted, thick] (v3) -- (v9);
\draw[black, dotted, thick] (v4) -- (v10);
\draw[black, thick] (v7) -- (v8);
\draw[black, thick] (v8) -- (v9);
\draw[black, thick] (v8) -- (v10);

\node at (5.9,4.5) {$\xleftarrow{\makebox[0.62cm]{}}$ even paths $\xrightarrow{\makebox[0.62cm]{}}$};

\end{tikzpicture}
\end{center}
\vspace{-0.35cm}
\caption{Bicycles and thetas (the dotted lines represent even paths)}
\label{fig:bicycle_theta}
\end{figure}
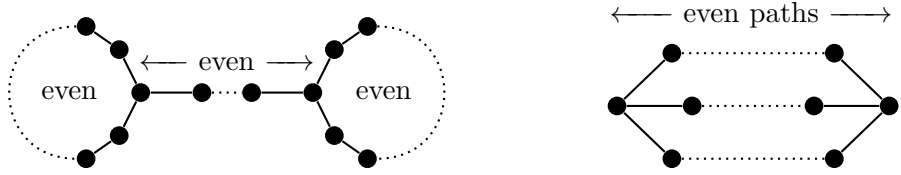

In what follows, when $G = L(B)$, for an edge $e \in E(B)$, we let $\ell(e)$ denote the vertex of $G$ that corresponds to $e$ in the line graph, and for $F \subseteq E(B)$, we let $\ell(F) = \{\ell(e) : e \in F\}$. The following theorem first appeared in the junior thesis of Andrei Graur \cite{Andrei}. Here, we provide a simpler proof.

\begin{theorem}
\label{thm:not_line_graph_of_bip}
If $(G,Z)$ is a suspect, then $G$ is not the line graph of a bipartite multigraph.
\end{theorem}

\begin{proof}
Let $(G,Z)$ be a suspect and assume that $G$ is the line graph of a bipartite multigraph $B$, i.e., $G = L(B)$. By Observation \ref{obs:parallel_edges}, $B$ is a bipartite graph as $G$ does not admit twins. Also, $B$ is connected (and so it has a unique bipartition) since otherwise $G$ is not connected, contrary to Theorem \ref{thm:suspect_starters}. Moreover, since $G$ is an innocent graph, $B$ is harmless.

Since the maximal cliques in $G$ correspond to stars in $B$ with centers of degree at least two, a strong stable set in $G$ corresponds to a matching in $B$ that covers all the vertices of degree at least two. Let us call such a matching of $B$ \emph{suitable}. Also, for a bipartite graph $H$, let $P(H)$ denote the set of vertices of $H$ with degree at least two.

Let $Z = \ell(Z_B)$. Then, $Z_B$ is a matching in $B$ as $Z$ is a stable set in $G$. Finding a strong stable set in $G$ containing $Z$ is the same as finding a suitable matching in $B$ containing $Z_B$. Let us call an edge $e$ of $B$ \emph{singular} if one of the endpoints of $e$ has degree one in $B$. Observe that if $v \in Z$, then $v = \ell(e)$ for some singular edge $e$ of $B$. By Theorem \ref{thm:suspect_starters}, every simplicial vertex of $G$ is in $Z$, and so, $Z$ is the set of simplicial vertices of $G$. Therefore, $Z_B \subseteq E(B)$ is precisely the set of singular edges in $B$. 

\vspace{-0.35cm}

\begin{equation}\label{eqn:singular_edges}
\longbox{{\it For every two singular edges $e_1, e_2$ in $B$, their degree-one endpoints are on different sides of the bipartition of $B$. Thus, $|Z_B| \leq 2$.}} \tag{$4.8.1$}
\end{equation}

\begin{proof}
Otherwise, there is an odd path in $G$ from $\ell(e_1)$ to $\ell(e_2)$, contrary to $\ell(e_1), \ell(e_2) \in Z$.
\end{proof}

\vspace{-0.8cm}

\begin{equation}\label{eqn:deg_3_stable}
\longbox{{\it The vertices in $B$ with degree at least three form a stable set.}} \tag{$4.8.2$}
\end{equation}

\begin{proof}
Let $u, v \in V(B)$ have degree at least three and assume $u$ and $v$ are adjacent. Let $e = uv$. Note that $e \notin Z_B$. Let $\hat B = B \setminus e$. Then, $deg_{\hat B}(u) \geq 2$ and $deg_{\hat B}(v) \geq 2$, and thus $P(B) = P(\hat B)$. Now, by the minimality of $G$, there exists a suitable matching $M$ of $\hat B$ containing $Z_B$, and since $P(B) = P(\hat B)$, $M$ is a suitable matching of $B$, a contradiction.
\end{proof}

\vspace{-0.8cm}

\begin{equation}\label{eqn:deg_3_exists}
\longbox{{\it $B$ has a vertex of degree at least three.}} \tag{$4.8.3$}
\end{equation}

\begin{proof}
Let $B$ have no vertex degree of at least three. Then, as $B$ is connected, $B$ is either an even cycle $c_1 \d c_2 \d \dots \d c_k \d c_1$ where $k \geq 4$ even, or an odd path $p_1 \d p_2 \d \dots \d p_m$ where $m \geq 2$ even. In the former case, $Z_B = \emptyset$, and $\{c_1c_2, c_3c_4, \dots, c_{k-1}c_k\}$ is a suitable matching in $B$. In the latter case, $Z_B = \{p_1p_2, p_{m-1}p_m\}$, and $\{p_1p_2, p_3p_4, \dots, p_{m-1}p_m\}$ is a suitable matching in $B$ that contains $Z_B$.
\end{proof}

By \eqref{eqn:deg_3_exists}, there exists a vertex $v \in V(B)$ of degree at least three. By \eqref{eqn:singular_edges} and \eqref{eqn:deg_3_stable}, $v$ has a neighbor $u$ of degree two in $B$. Let $N_B(u) = \{v, w\}$. We claim that $N_B(v) \cap N_B(w) = \{u\}$. Let $x$ be a common neighbor of $v$ and $w$ in $B$ such that $x \neq u$. By \eqref{eqn:deg_3_stable}, $deg_B(x) = 2$. Let $A = \{\ell(xw), \ell(wu)\}$ and $B = \{\ell(uv), \ell(vx)\}$. Then, it is straightforward to check that $(A, B)$ is a proper coherent homogeneous pair in $G$, contrary to Theorem \ref{thm:no_prop_coh_W}. Therefore, $N_B(v) \cap N_B(w) = \{u\}$.

For $x \in \{v, w\}$, let $E_x = \{e \in E(B) : e \text{ is incident with } x\} \setminus \{xu\}$. We proved that $E_v \cap E_w = \emptyset$. We have $|E_v| \geq 2$ since $deg_B(v) \geq 3$, but possibly $E_w = \emptyset$. Let $\hat B$ be the graph obtained from $B$ by deleting $u$, and identifying the vertices $v, w$. It is easy to see that $\hat B$ is still bipartite and connected.

\vspace{-0.5cm}

\begin{equation}\label{eqn:B_hat_harmless}
\mbox{{\it $\hat B$ is harmless.}} \tag{$4.8.4$}
\end{equation}
\begin{proof}
Suppose $\hat B$ is not harmless. Since line graphs of bipartite graphs do not contain odd holes or antiholes of length at least seven, and since an antihole of length six is also an odd prism, it follows that $\hat B$ contains either a bicycle or a theta as a subgraph (not necessarily induced). We will show that $B$ too contains either a bicycle or a theta, leading to a contradiction.

Let $t \in V(\hat B)$ be the vertex that is obtained by identifying the vertices $v, w$ of $B$, and let $F$ be the subgraph of $\hat B$ that is either a bicycle of a theta. Since $B$ is harmless, $t \in V(F)$, $N_F(t) \not\subseteq N_B(v)$ and $N_F(t) \not\subseteq N_B(w)$. Also, every $x \in N_F(t)$ has a neighbor in $\{v, w\}$ in $B$. Assume first that $|N_F(t)| = 2$, say $N_F(t) = \{a, b\}$. Then, replacing the path $a \d t \d b$ by a path from $a$ to $b$ with interior in $\{v, u, w\}$ gives a forbidden subgraph of the same kind in $B$, a contradiction. Next, assume that $|N_F(t)| = 3$, say $N_F(t) = \{a, b, c\}$. We may assume that $a, b$ are adjacent to $v$ in $B$, and $c$ is adjacent to $w$ in $B$. Observe that $B$ is obtained from $\hat B$ by subdividing the edge $tc$ twice. As this operation preserves the parity of the paths and cycles appearing in a bicycle or in a theta, this yields a forbidden subgraph of the same kind in $B$, a contradiction. See Figure \ref{fig:theta_in_B} for an illustration when $F$ is a theta.



\vspace{-0.2cm}

\begin{figure}[h]
\begin{center}
\begin{tikzpicture}[scale=0.23]

\node[label=left:{\small{$t$}}, inner sep=2.5pt, fill=black, circle] at (-2, 0)(v5){}; 
\node[label=above:{\small{$c$}}, inner sep=2.5pt, fill=black, circle] at (3, 3.2)(v3){}; 
\node[inner sep=2.5pt, fill=black, circle] at (3, 0)(v6){}; 
\node[label=below:{\small{$a$}}, inner sep=2.5pt, fill=black, circle] at (3, -3.2)(v4){};
\node[inner sep=2.5pt, fill=black, circle] at (8, 0)(v8){};

\node at (2.2, -1) {\small{$b$}};

\draw[black, thick] (v5) -- (v3);
\draw[black, thick] (v5) -- (v4);
\draw[black, thick] (v5) -- (v6);
\draw[black, thick] (v8) -- (v3);
\draw[black, thick] (v8) -- (v4);
\draw[black, thick] (v8) -- (v6);

\end{tikzpicture}
\hspace{2.5cm}
\begin{tikzpicture}[scale=0.23]

\node[label=left:{\small{$v$}}, inner sep=2.5pt, fill=black, circle] at (-2, 0)(v5){}; 
\node[label=above:{\small{$u$}}, inner sep=2.5pt, fill=black, circle] at (0, 3.2)(v3){}; 
\node[inner sep=2.5pt, fill=black, circle] at (3, 0)(v6){}; 
\node[label=below:{\small{$a$}}, inner sep=2.5pt, fill=black, circle] at (3, -3.2)(v4){};
\node[inner sep=2.5pt, fill=black, circle] at (8, 0)(v8){};

\node[label=above:{\small{$w$}}, inner sep=2.5pt, fill=black, circle] at (3, 3.2)(v9){}; 
\node[label=above:{\small{$c$}}, inner sep=2.5pt, fill=black, circle] at (6, 3.2)(v10){}; 

\node at (2.2, -1) {\small{$b$}};

\draw[black, thick] (v5) -- (v3);
\draw[black, thick] (v5) -- (v4);
\draw[black, thick] (v5) -- (v6);
\draw[black, thick] (v8) -- (v4);
\draw[black, thick] (v8) -- (v6);
\draw[black, thick] (v3) -- (v9);
\draw[black, thick] (v9) -- (v10);
\draw[black, thick] (v10) -- (v8);

\end{tikzpicture}
\end{center}
\vspace{-0.5cm}
\caption{Thetas in $\hat B$ and in $B$}
\label{fig:theta_in_B}
\end{figure}
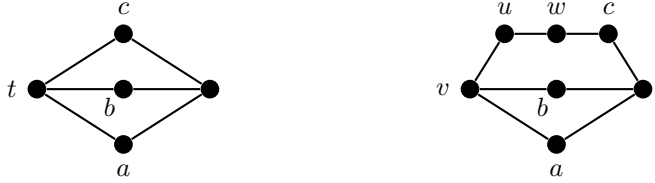

Hence, $|N_F(t)| = 4$ and $F$ is a bicycle where the even path connecting the two cycles $C_1$, $C_2$ is of length zero. Let $N_F(t) = \{a, b, c, d\}$, say $N_F(t) \cap V(C_1) = \{a, b\}$ and $N_F(t) \cap V(C_2) = \{c, d\}$. There are three cases (up to symmetry). If $a, b, c$ are adjacent to $v$ in $B$ and $d$ is adjacent to $w$ in $B$, then $B$ is obtained from $\hat B$ by subdividing the edge $td$ twice. If $a, b$ are adjacent to $v$ in $B$ and $c, d$ are adjacent to $w$ in $B$, then $B$ is obtained from $\hat B$ by replacing the vertex $t$ with the two-edge path $v \d u \d w$. In both cases, $B$ contains a bicycle. If $a, c$ are adjacent to $v$ in $B$ and $b, d$ are adjacent to $w$ in $B$, then $B$ contains a theta with the three even paths $v \d a \d C_1 \d b \d w$, $v \d u \d w$, and $v \d c \d C_2 \d d \d w$, a contradiction.
\end{proof}

By \eqref{eqn:singular_edges}, not both $E_v$ and $E_w$ contain a singular edge. Let $Z_{\hat B} = Z_B \setminus \{uw\}$ if $deg_B(w) = 1$ (i.e., if $E_w = \emptyset$), and $Z_{\hat B} = Z_B$ otherwise. Now, by the minimality of $G$, there exists a suitable matching $\hat M$ in $\hat B$ containing $Z_{\hat B}$. Since all the edges in $E_v \cup E_w$ share a common endpoint in $\hat B$, we have $|\hat M \cap (E_v \cup E_w)| = 1$. If $|\hat M \cap E_v| = 1$, let $M = \hat M \cup \{uw\}$, and if $|\hat M \cap E_w| = 1$, let $M = \hat M \cup \{uv\}$. Then, $M$ is a suitable matching in $B$ containing $Z_B$, a contradiction.
\end{proof}

Let $G$ be a claw-free graph such that $G$ is an augmentation of a graph $H$. Let $ab$ be a flat edge of $H$ that is augmented by a cobipartite graph $(A,B)$ in obtaining $G$. Being a flat edge, $a$ and $b$ do not have a common neighbor in $H$. Let $C = N_{H}(a) \setminus \{b\}$, $D = N_{H}(b) \setminus \{a\}$, and let $F$ be the common non-neighbors of $a$ and $b$ in $H$. In $G$, $A$ is complete to $C$ and anticomplete to $D \cup F$; and $B$ is complete to $D$ and anticomplete to $C \cup F$. The sets $C$ and $D$ are cliques since otherwise $G$ is not claw-free. We say that $(A, B, C, D, F)$ is an \emph{augmentation partition} of $G$ by the edge $ab$.

Let $x_1y_1, \dots, x_hy_h$ be $h$ pairwise disjoint flat edges of $H$ that are augmented by pairwise disjoint cobipartite graphs $(X_1,Y_1), \dots, (X_h, Y_h)$, respectively, to obtain $G$. We say $G$ is a \emph{smooth augmentation} of $H$ if for $i=1, \dots, h$, every $x \in X_i$ has a neighbor in $Y_i$, and every $y \in Y_i$ has a neighbor in $X_i$.

\begin{theorem}
\label{thm:not_augm_line_graph_of_bip}
If $(G,Z)$ is a suspect, then $G$ is not a smooth augmentation of the line graph of a bipartite multigraph.
\end{theorem}

\begin{proof}
Let $(G,Z)$ be a suspect and assume $G$ is a smooth augmentation of the line graph of a bipartite multigraph $J$. We let $H = L(J)$, and we follow the notation given before the statement of the theorem. If $|X_i| = |Y_i| = 1$ for every $i=1, \dots, h$, then, $G$ is the line graph of a bipartite multigraph, contrary to Theorem \ref{thm:not_line_graph_of_bip}. Thus, we may assume that $|X_1 \cup Y_1| >  2$. Let $(A = X_1, B = Y_1, C, D, F)$ be an augmentation partition of $G$ by the edge $x_1y_1$. As $G$ is a smooth augmentation of $H$, every vertex of $A$ has a neighbor in $B$, and every vertex of $B$ has a neighbor in $A$. Also, $A \cup C$ and $B \cup D$ are cliques.

We show that $Z \cap (A \cup B) = \emptyset$. Assume $Z \cap A \neq \emptyset$, say $z \in Z \cap A$. Then, $C = \emptyset$ since $z$ is a simplicial vertex and $z$ has a neighbor in $B$. If $F = \emptyset$, then $G$ is a cobipartite graph, and by Lemma \ref{lem:cobipartite}, $G$ has a strong stable set that contains $Z$, a contradiction. Thus, $F \neq \emptyset$. Then, $D \neq \emptyset$ since $G$ is connected by Theorem \ref{thm:suspect_starters}. But now, $(A \cup B, D \cup F)$ is a $1$-join in $G$, contrary to Theorem \ref{thm:no_1_join}. This proves that $Z \cap (A \cup B) = \emptyset$.

Assume first that there is no square in $G | (A \cup B)$. As every vertex of $A$ has a neighbor in $B$ and every vertex of $B$ has a neighbor in $A$, by Lemma \ref{lem:no_C4_order} and using the symmetry between $A$ and $B$, there exist vertices $u \in A$ and $v \in B$ such that $u$ is complete to $B$ and $v$ is complete to $A$. Let $M_A = A \setminus u$, $M_B = B \setminus v$, and $M = M_A \cup M_B$. Let $\tilde G = G \setminus M$. By the minimality of $G$, there exists a strong stable set $\tilde S$ in $\tilde G$ containing $Z$. We claim that $\tilde S$ is a strong stable set in $G$. Since $\{u, v\}$ is a maximal clique in $\tilde G$, either $u \in \tilde S$ or $v \in \tilde S$. Let $K$ be a maximal clique in $G$ such that $\tilde S \cap K = \emptyset$. Since $K$ is not a maximal clique of $\tilde G$, we deduce that $K \cap M \neq \emptyset$. If $K \cap M_A \neq \emptyset$ and $K \cap M_B = \emptyset$, then $K = A \cup C$. If $K \cap M_A = \emptyset$ and $K \cap M_B \neq \emptyset$, then $K = B \cup D$. If $K \cap M_A \neq \emptyset$ and $K \cap M_B \neq \emptyset$, then $u, v \in K$. So, in each case, $\tilde S \cap K \neq \emptyset$, a contradiction. This proves that $G | (A \cup B)$ contains a square.

Let $H = a_1 \d b_1 \d b_2 \d a_2 \d a_1$ be a square in $G | (A \cup B)$ with $a_1, a_2 \in A$ and $b_1, b_2 \in B$. Notice that no vertex of $V(G) \setminus (A \cup B)$ is mixed on $\{a_1, a_2\}$, and no vertex of $V(G) \setminus (A \cup B)$ is mixed on $\{b_1, b_2\}$. Thus, by Lemma \ref{lem:get_W_join} applied to the pair of cliques $(A, B)$, $G$ admits a proper coherent $W$-join, a contradiction to Theorem \ref{thm:no_prop_coh_W}.
\end{proof}

\section{The proof of the main theorem}\label{sec:proof}

In this section we present the proof of Theorem \ref{thm:stronger}, which we restate here one last time.

\begin{theorem}
\label{thm:stronger3}
Let $G$ be a claw-free innocent graph and let $Z$ be a consistent set of safe vertices in $G$. Then $G$ has a strong stable set that contains $Z$.
\end{theorem}

\begin{proof}
Suppose the theorem is false, and let $(G,Z)$ be a suspect. By Theorem \ref{thm:suspect_starters}, $G$ is connected, is not a complete graph, has no twins, and every simplicial vertex of $G$ is in $Z$. Let $G' = G \setminus Z$.

\vspace{-0.6cm}

\begin{equation}
\longbox{{\it $G'$ is connected.}}
\label{eq:G'_is_connected}
\end{equation}

\begin{proof}
Assume that $G'$ is not connected. As the vertices in $Z$ are simplicial in $G$, they do not have neighbors in two different components of $G'$. Since $Z$ is stable, $G$ is not connected, a contradiction.
\end{proof}

\vspace{-0.8cm}

\begin{equation}
\longbox{{\it $G'$ is not a complete graph.}}
\label{eq:G'_is_not_complete}
\end{equation}

\begin{proof}
Assume that $G'$ is a complete graph. Then $G' \neq G$. As there is no three-edge path between two distinct vertices of $Z$, by Lemma \ref{lem:no_C4_order}, we can order the vertices of $Z$, say $v_1, v_2, \dots, v_p$, in such a way that if $i \leq j$, then $N(v_i) \subseteq N(v_j)$. This implies that either $v_p$ is complete to $V(G')$ or there is a vertex $u \in V(G')$ such that $u$ is anticomplete to $Z$. In the first case $Z$, and in the second case $Z \cup \{u\}$ is a strong stable set in $G$ containing $Z$, a contradiction.
\end{proof}

\vspace{-0.8cm}

\begin{equation}
\longbox{{\it $G'$ is not a peculiar graph.}}
\label{eq:peculiar}
\end{equation}

\begin{proof}
Suppose $G'$ is peculiar. If $Z = \emptyset$, then $G = G'$ and $G$ is an innocent peculiar graph, and by Lemma \ref{lem:peculiar}, $G$ has a strong stable set that contains $Z$, a contradiction. Hence, $Z \neq \emptyset$. Now, let $z \in Z$. By Observation \ref{obs:simplicial_clique}, the set $N(z)$ is a simplicial clique in $G\setminus z$. Since $Z$ is a stable set, $N(z) \subseteq V(G')$, and so $N(z)$ is a simplicial clique in $G'$ as well, contrary to Lemma \ref{lem:peculiar_no_simp_clique}.
\end{proof}

\vspace{-0.8cm}

\begin{equation}
\longbox{{\it $G'$ does not admit a clique cutset.}}
\label{eq:clique_cutset}
\end{equation}

\begin{proof}
If $G'$ admits a clique cutset, then by Observation \ref{obs:simplicial_internal}, $G$ admits an internal clique cutset. As $G$ is connected and does not admit twins, Theorem \ref{thm:claw_free_internal_cc_proper} implies that either $G$ is a linear interval graph, or $G$ admits a coherent proper $W$-join, or $G$ admits a $1$-join, contrary to Theorems \ref{thm:not_linear_interval}, \ref{thm:no_prop_coh_W}, \ref{thm:no_1_join}.
\end{proof}

\vspace{-0.8cm}

\begin{equation}
\longbox{{\it $G'$ is not a cobipartite graph.}}
\label{eq:G'_is_not_cobipartite}
\end{equation}

\begin{proof}
Assume that $G'$ is a cobipartite graph $(A,B)$. By Lemma \ref{lem:cobipartite}, $Z \neq \emptyset$. As there is no three-edge path in $G$ between two vertices of $Z$, by Lemma \ref{lem:no_C4_order}, there is a vertex $z_A \in Z$ such that $N(z) \cap A \subseteq N(z_A) \cap A$ for every $z \in Z$. Similarly, there is a vertex $z_B \in Z$ such that $N(z) \cap B \subseteq N(z_B) \cap B$ for every $z \in Z$. Possibly $z_A = z_B$. Let $A_1 = A \cap N(z_A)$, $A_2 = A \setminus A_1$, and $B_1 = B \cap N(z_B)$, $B_2 = B \setminus B_1$.

\vspace{-0.6cm}

\begin{equation}\label{eqn:star}
\longbox{{\it For $a_1 \in A_1$, $N(a_1) \cap B_2$ is complete to $A_2$.}} \tag{$5.1$}
\end{equation}

\begin{proof}
Suppose not. Let $b_2 \in N(a_1) \cap B_2$ and let $a_2 \in A_2$ be a non-neighbor of $b_2$, then $\{a_1, a_2, b_2, z_A\}$ is a claw in $G$, a contradiction.
\end{proof}

If $G'$ contains no square, then by Lemma \ref{lem:no_C4_order}, there is a vertex $a \in A$ such that $N(a) \cap B \subseteq N(a') \cap B$ for every $a' \in A$, now $N_{G'}(a)$ is a clique cutset of $G'$, a contradiction to \eqref{eq:clique_cutset}. Therefore, there is a square $a_1 \d b_1 \d b_2 \d a_2 \d a_1$ in $G'$. By Observation \ref{obs:path_to_hole}, $a_1, a_2 \in A_2$ and $b_1, b_2 \in B_2$. Suppose that $A_1 \neq \emptyset$ and  $B_1 \neq \emptyset$, let $a \in A_1$ and $b \in B_1$. By \eqref{eqn:star}, $a$ is anticomplete to $\{b_1, b_2\}$ and $b$ is anticomplete to $\{a_1, a_2\}$. Since $G$ contains no odd prism, we deduce that $a$ is non-adjacent to $b$. Consequently, $z_Ab$ and $z_Ba$ are non-edges as $z_A, z_B$ are simplicial vertices. But now, $z_A \d a \d a_1 \d b_1 \d b \d z_B$ is a path of length five, contrary to the fact that $Z$ is consistent. Thus, we may assume that $B_1 = \emptyset$, and so $Z$ is anticomplete to $B$.

Now, let $ab$ be a cosimplicial non-edge in $G | (A_2 \cup B)$ where $a \in A$, $b \in B$. Then, since $S = Z \cup \{a, b\}$ is not a strong stable set in $G$, there exists a maximal clique $C$ such that $S \cap C = \emptyset$. Since $C$ is maximal, it contains a non-neighbor $b'$ of $b$, and a non-neighbor $a'$ of $a$. Then $b' \in A$ and $a' \in B$. Since $a'b'$ is an edge and $ab$ is a cosimplicial non-edge in $G | (A_2 \cup B)$, it follows that $b' \in A_1$. But $b'$ is complete to $\{a, a'\}$ contrary to \eqref{eqn:star}. This completes the proof of \eqref{eq:G'_is_not_cobipartite}.
\end{proof}

Being an induced subgraph of $G$, the graph $G'$ is innocent, and therefore perfect. By Theorem \ref{thm:Chvatal_Sbihi}, $G'$ is either peculiar, or elementary, or admits a clique cutset. By \eqref{eq:peculiar} and \eqref{eq:clique_cutset}, $G'$ is elementary. Then by Theorem \ref{thm:Maffray_Reed}, $G'$ is an augmentation of the line graph of a bipartite multigraph. Let $G''$ be the line graph of a bipartite multigraph $B''$, i.e., $G'' = L(B'')$, and let $G'$ be an augmentation of $G''$. Note that $B''$ is connected since otherwise $G''$ is not connected, and so $G'$ is not connected, contrary to \eqref{eq:G'_is_connected}.

Let $\ell(a) \ell(b)$ be a flat edge of $G''$ where $a, b \in E(B'')$. If $a, b$ are parallel edges, then no other edge in $B''$ shares a vertex with $a$ or $b$, and since $B''$ is connected, it follows that $E(B'') = \{a, b\}$, and so $G'$ is a cobipartite graph, contrary to \eqref{eq:G'_is_not_cobipartite}. Hence, if $\ell(a) \ell(b)$ is a flat edge of $G''$, then there exist three distinct vertices $u, v, w \in V(B'')$ such that $N_{B''}(v)=\{u, w\}$ and $a = vu$, $b = vw$. Conversely, if such three vertices $u, v, w \in V(B'')$ exist, then $\ell(a) \ell(b)$ is a flat edge in $G''$. So, for every flat edge $\ell(a) \ell(b)$ of $G''$, there is a vertex $v \in V(B'')$ with $deg_{B''}(v) = 2$ as described above. We call $v$ \emph{special} if the flat edge $\ell(a) \ell(b)$ of $G''$ is augmented in obtaining $G'$, and \emph{non-special} otherwise.

Thus, $G'$ is obtained from $B''$ as follows. Every edge $e \in E(B'')$ corresponds to a clique $X_e$ in $G'$. Let $e_1, e_2 \in E(B'')$. If $e_1, e_2$ are distinct, then $X_{e_1} \cap X_{e_2} = \emptyset$. If $e_1, e_2$ have no common endpoint, then $X_{e_1}$ is anticomplete to $X_{e_2}$ in $G'$. If $e_1, e_2$ are incident with a common vertex of degree at least three or incident with a common non-special vertex, then $X_{e_1}$ is complete to $X_{e_2}$ in $G'$. Lastly, if $e_1, e_2$ are incident with a common special vertex, then $G' | (X_{e_1}, X_{e_2})$ is cobipartite and $|X_{e_1} \cup X_{e_2}| > 2$.

\vspace{-0.4cm}

\begin{equation}
\longbox{{\it Let $e_1, e_2 \in E(B'')$ be incident with a special vertex $v \in V(B'')$. Then, in $G'$, every vertex of $X_{e_1}$ has a neighbor in $X_{e_2}$, and every vertex of $X_{e_2}$ has a neighbor in $X_{e_1}$.}}
\label{eq:G'_smooth}
\end{equation}

\begin{proof}
Let $(X_{e_1}, X_{e_2}, C, D, F)$ be the augmentation partition of $G'$ by the flat edge $\ell(e_1) \ell(e_2) \in E(G'')$. Let $x \in X_{e_1}$ be anticomplete to $X_{e_2}$ in $G'$. Then, $(X_{e_1} \setminus x) \cup C$ is a clique cutset in $G'$ separating $x$ from $X_{e_2} \cup D \cup F$, a contradiction to \eqref{eq:clique_cutset}.
\end{proof}

\vspace{-0.7cm}

\begin{equation}
\longbox{{\it Let $e_1 \in E(B'')$ be incident with a special vertex $v$ in $B''$, and let $z \in Z$. Then, $z$ is anticomplete to $X_{e_1}$ in $G$.}}
\label{eq:z_special}
\end{equation}

\begin{proof}
Let $e_2$ be the other edge that is incident to $v$ in $B''$, and let $(X_{e_1}, X_{e_2}, C, D, F)$ be the augmentation partition of $G'$ by the flat edge $\ell(e_1) \ell(e_2) \in E(G'')$. If $C = \emptyset$, then $X_{e_2}$ is a clique cutset in $G'$ separating $X_{e_1}$ from $D \cup F$, contrary to \eqref{eq:clique_cutset}, unless $D \cup F = \emptyset$, in which case $G'$ is cobipartite, contrary to \eqref{eq:G'_is_not_cobipartite}. Therefore $C \neq \emptyset$, and similarly $D \neq \emptyset$. Now, assume that there is no path $P$ from $C$ to $D$ with $P^* \subseteq F$. Then $F = F_C \cup F_D$ such that $C \cup F_C$ is anticomplete to $D \cup F_D$. But then $X_{e_1}$ is a clique cutset in $G'$ contrary to \eqref{eq:clique_cutset}. Hence, there is a path $P$ from $C$ to $D$ with $P^* \subseteq F$.

Assume that $z$ has a neighbor $x$ in $X_{e_1}$. By \eqref{eq:G'_smooth}, $x$ has a neighbor $y \in X_{e_2}$. Now, $V(P) \cup \{x, y\}$ is a hole, and $z$ has a neighbor in this hole, contrary to Observation \ref{obs:path_to_hole}.
\end{proof}

\vspace{-0.7cm}

\begin{equation}
\longbox{{\it For every $z \in Z$, there exists a non-special vertex $t \in V(B'')$ such that every $e \in E(B'')$ with $X_e$ not anticomplete to $z$ is incident with $t$.}}
\label{eq:there_exists_t}
\end{equation}

\begin{proof}
Let $z \in Z$ have neighbors in $X_{e_1}, \dots, X_{e_k} \subseteq V(G')$ where $e_i \in E(B'')$ for $i=1,\dots,k$. Since $z$ is simplicial, there are edges between $X_{e_i}$ and $X_{e_j}$ for $i \neq j$, and so for every $i, j$, the edges $e_i, e_j$ share an endpoint in $B''$. Since $B''$ is bipartite, there exists $t \in V(B'')$ such that $e_1, \dots, e_k$ are incident with $t$. Also, by \eqref{eq:z_special}, $t$ is non-special.
\end{proof}

By \eqref{eq:G'_smooth}, it follows that $G'$ is a smooth augmentation of $G''$. Then, by Theorem \ref{thm:not_augm_line_graph_of_bip}, $Z \neq \emptyset$. Let $Z = \{z_1, \dots, z_k\}$. Let $G_0 = G'$, $G_k = G$, and for $i=0, 1, \dots, k-1$, let $G_i = G_{i+1} \setminus z_{i+1}$. Let $i$ be minimum such that $G_{i+1}$ is not a smooth augmentation of the line graph of a bipartite multigraph. Then, $0 \leq i \leq k-1$. Let $G_i$ be a smooth augmentation of $L(B_i)$ where $B_i$ is a bipartite multigraph.

Note that $B''$ and $B_i$ have the same set of special vertices. By \eqref{eq:z_special}, if $a \in E(B_i)$ is incident with a special vertex in $B_i$, then $X_a \cap \{z_1, \dots, z_i\} = \emptyset$ and $z_{i+1}$ is anticomplete to $X_a \subseteq V(G_i)$. Let $t \in V(B_i)$ be as in \eqref{eq:there_exists_t} for $z_{i+1}$, and let $I = \{a_1, \dots, a_m\}$ denote the set of edges incident with $t$ in $B_i$. We can partition $I$ as $I = Y \cup N$ ($Y$ for yes, $N$ for no) where $Y = \{a_j : z_{i+1} \text{ has a neighbor in } X_{a_j} \subseteq V(G_i)\}$ and $N = \{a_j : z_{i+1} \text{ is anticomplete to } X_{a_j} \subseteq V(G_i)\}$. Since $G_{i+1}$ is not a smooth augmentation of the line graph of a bipartite multigraph, $N \neq \emptyset$.

Let $a_k \in Y$ and $a_l \in N$. If there is an edge $e$ in $E(B_i) \setminus I$ sharing a vertex with $a_k$, then, $\{\ell(a_k), \ell(a_l), \ell(e), z_{i+1}\}$ is a claw in $G$. Hence, no edge in $E(B_i) \setminus I$ shares a vertex with an edge in $Y$.

By \eqref{eq:G'_is_not_complete}, $E(B'') \not\subseteq I$ since $G'$ is not a complete graph. Also, $\ell(Y) \cap Z = \emptyset$ since $Z$ is a stable set. But now, $\ell(N)$ is a clique cutset in $G_i$ separating $\ell(Y)$ from $\ell(E(B_i) \setminus I)$. It follows that $\ell(N) \setminus Z$ is a clique cutset in $G'$ separating $\ell(Y)$ from $\ell(E(B'') \setminus I)$. Hence, we get a contradiction to \eqref{eq:clique_cutset}. This completes the proof of Theorem \ref{thm:stronger3}.
\end{proof}

\section{Acknowledgment}
The authors would like to thank Sophie Spirkl for many helpful discussions.


\end{document}